\documentclass[reqno, 11pt]{jams-l}

\usepackage{amssymb, amsmath}
\usepackage{amsthm}
\usepackage[colorlinks,citecolor=red,pagebackref,hypertexnames=false]{hyperref}
\usepackage{color}
\usepackage{esint}
\usepackage{bbm}
\usepackage{graphicx}
\usepackage{etoolbox}
\usepackage{mathrsfs}
\usepackage{enumitem}
\usepackage{mathrsfs}
\usepackage{todonotes}
\usepackage{dsfont}
\usepackage{comment}

\usepackage[left=3.3cm, right=3.3cm, top=2.7cm, bottom=2.7cm]{geometry}

\def\dM{{\mathcal{M}}}
\def\dS{{\mathcal{S}}}
\def\dT{\mathcal{T}}

\newcommand{\cc}{\mathsf{c}}
\newcommand{\Ww}{\mathsf{W}}

\def\ve{\epsilon} % needed to  unite between \epsiln and \ve because we were not consistent

\def\lec{\lesssim}

\DeclareMathOperator{\diam}{diam}
\def\Lip{\mathop\mathrm{Lip}} 						%lipschitz constant 
\def\dist{\mathop\mathrm{dist}} 						%distance
\def\Int{\mathop\mathrm{Int}} 						%Interior
\newcommand{\ps}[1]{\left( #1 \right)}

\def\XXint#1#2#3{{\setbox0=\hbox{$#1{#2#3}{\int}$ }
\vcenter{\hbox{$#2#3$ }}\kern-.58\wd0}}

\makeatletter
  \def\th@citetheorem{
  \thm@headfont{\itshape} % Heading font is italic
  \thm@notefont{\normalfont} % Note is same as heading
  \itshape% Regular text is also italic
}
\makeatother

\makeatletter
  \def\th@maintheorem{
  \thm@headfont{\bfseries} % Heading font is italic
  \itshape% Regular text is also italic
}
\makeatother

\theoremstyle{maintheorem}
\newtheorem{maintheorem}{Theorem}

\theoremstyle{theorem}
\newtheorem{theorem}{Theorem}[section]
\newtheorem{lemma}[theorem]{Lemma}

\theoremstyle{citetheorem}
\newtheorem{citetheorem}[theorem]{Theorem}

\theoremstyle{definition}
\newtheorem{definition}[theorem]{Definition}

\theoremstyle{remark}
\newtheorem{remark}[theorem]{Remark}

\numberwithin{equation}{section}

\newcommand{\R}{\mathbb{R}}
\newcommand{\N}{\mathbb{N}}
\newcommand{\Z}{\mathbb{Z}}
\newcommand{\C}{\mathbb{C}}

\newcommand{\lip}{\mathrm{Lip}}
\newcommand{\hd}{\mathcal{H}^d}

\newcommand{\hdc}{\mathcal{H}^d_\infty}

\newcommand{\B}{\mathbb{B}}

\newcommand{\betae}[1]{\beta_E^{d,#1}}

\newcommand{\cubes}{\mathcal{D}}

\newcommand{\chara}{\mathbbm{1}}

\newcommand{\Stop}{\mathop\mathrm{Stop}}
\newcommand{\Next}{\mathop\mathrm{Next}}
\newcommand{\tree}{\mathop\mathrm{Tree}}
\newcommand{\Top}{\mathop\mathrm{Top}}

\newcommand{\wt}{\widetilde}

\newcommand\blfootnote[1]{%
  \begingroup
  \renewcommand\thefootnote{}\footnote{#1}%
  \addtocounter{footnote}{-1}%
  \endgroup
}

\numberwithin{equation}{section}
\theoremstyle{plain}

\newtheorem{corollary}[theorem]{Corollary}
\newtheorem{proposition}[theorem]{Proposition}

\title[Cones and paraboloids]{Cone and paraboloid points of arbitrary subsets of Euclidean space}
\author{Matthew Hyde and Michele Villa}

\newcommand{\Addresses}{{% additional braces for segregating \footnotesize
  \bigskip
  \footnotesize
  
     M. Hyde, \textsc{School of Mathematics,
University of Edinburgh,
JCMB,
Kings Buildings,
Mayfield Road
Edinburgh,
EH9 3JZ,
Scotland}\par\nopagebreak
  \textit{E-mail address}: \texttt{m.hyde@ed.ac.uk}

  M. Villa, \textsc{School of Mathematics and Statistics, University of Jyv\"{a}skyl\"{a}}\par\nopagebreak
  \textit{E-mail address}: \texttt{mvilla@jyu.fi}

}}

\begin{document}
\maketitle

\begin{center}

\begin{minipage}[c][][r]{300pt}
\begin{small}
\textsc{Abstract.} In this paper we characterise cone points of arbitrary subsets of Euclidean space. Given $E \subset \R^n$, $x \in E$ is a cone point of $E$ if and only if
\begin{align*}
    \int_{0}^1 \beta_{E}^{d,2}(B(x,r))^2 \frac{dr}{r} < \infty,
\end{align*}
up to a set of zero $d$-measure. The coefficients $\beta_E^{d,2}$ are a variation of the Jones coefficients. This is a high dimensional counterpart of a theorem of Bishop and Jones from 1994. We also prove similar results for $\alpha$-paraboloid points, which are the $C^{1,\alpha}$ rectifiability counterparts to cone points: $x \in E$ is an $\alpha$-paraboloid point if and only if
\begin{align*}
    \int_0^1 \frac{\overline{\beta}_{E}^{d,2}(B(x,r))^2}{r^{2\alpha}} \, \frac{dr}{r} < \infty
\end{align*}
up to a set of zero $d$-measure. Here, $\overline{\beta}^{d,2}_E$ is another variant of the Jones coefficients, introduced by Azzam and Schul.
\end{small}
\end{minipage}
\end{center}
\blfootnote{\textup{2010} \textit{Mathematics Subject Classification}: \textup{28A75}, \textup{28A12} \textup{28A78}.

\textit{Key words and phrases.} Rectifiability, tangent points, beta numbers, Hausdorff content.

M. H. is supported by The Maxwell Institute Graduate School in Analysis and its
Applications, a Centre for Doctoral Training funded by the UK Engineering and Physical
Sciences Research Council (grant EP/L016508/01), the Scottish Funding Council, Heriot-Watt University and the University of Edinburgh. M. V. is supported by the Academy of Finland via the project Incidences on Fractals, grant No. 321896.
}
\tableofcontents
\section{Introduction}

Let $E \subset \R^n$, and $d<n$ an integer. One says that $E$ is $d$-\textit{rectifiable} if there are $d$-dimensional Lipschitz graphs $\Gamma_i$, $i=1,2,...$, so that
\begin{align*}
    \hd\left(E \setminus \bigcup_{i} \Gamma_i \right) = 0.
\end{align*}
On the other hand, a set is said to be \textit{purely} $d$-\textit{unrectifiable} if $\hd(E \cap F) = 0$ for any $d$-rectifiable set $F$. Rectifiability is a central notion in geometric measure theory, and characterising rectifiable sets is one of its main objectives. 

Let $V\subset \R^n$ be a $d$-dimensional affine plane. For a point $x \in \R^n$, we define the \textit{truncated cone} with radius $r>0$ and aperture $\lambda>0$ by
\begin{align*}
    X(x, V, \lambda, r) := \left\{ y \in \R^n \, | \, |\Pi_{V^\perp}(x-y)| < \lambda |\Pi_V(x-y)| \right\} \cap B(x,r), 
\end{align*}
where $B(x,r)=\{|x-y| < r\}$. Also let $X^c(x, V, \lambda, r) := B(x,r) \setminus X(x, V, \lambda, r)$. One of the best known characterisations of rectifiable sets is the following (see \cite{mattila} Theorem 15.19). 
\begin{citetheorem}\label{t:federer}
Let $E \subset \R^n$ be so that $0< \hd(E)< \infty$. Then $E$ is $d$-rectifiable if and only if there exists a unique $d$-dimensional affine plane $V$ so that 
\begin{align}\label{e:federer-char}
    \lim_{r \to 0} \frac{ \hd(X^c(x, V, \lambda, r)\cap E)}{r^d} = 0 \mbox{ for } \hd-a.e. \, \, x \in \R^n.
\end{align}
\end{citetheorem}
A $d$-plane $V$ satisfying \eqref{e:federer-char} is called an \textit{approximate tangent plane}, while an $x$ for which such a $V$ exists, is an \textit{approximate tangent point}. Other classical examples are characterisations in terms of densities, tangent measures, or orthogonal projections. For proofs of these influential results, see the monograph \cite{mattila}. 
More recently, there has been a need to quantify the notion of rectifiability, due to its connection with singular integrals and analytic capacity. This, starting with the pioneering works of Jones \cite{jones90} and David and Semmes \cite{david-semmes91}, has lead to the study of connections between rectifiability properties of sets and measures and boundedness of geometric square functions (see e.g. \cite{bishop1994harmonic}, \cite{david1998unrectictifiable}, \cite{leger1999menger}, \cite{ntv}, \cite{azzam2015characterization}, \cite{jaye2019proof}).

 Given a set $E$, we say that $x \in E$ is a $d$-\textit{cone point} if there is a $d$-plane $V$ so that for all $\lambda>0$ we can find an $r>0$ such that $E \cap X^c(x, V, \lambda, r)= \emptyset$. It is easy to see that the subset of $E$ of cone points is $d$-rectifiable (\cite{mattila}, Lemma 15.13). 
 
A prototypical and influential result connecting geometric square functions to rectifiability is the following theorem by Bishop and Jones from 1994.
\begin{citetheorem}[{\cite{bishop1994harmonic}, Theorem 2}]\label{t:bishop-jones}
Let $\Gamma \subset \C$ be a Jordan curve. Then, up to set of $\mathcal{H}^1$ measure zero, $x \in \Gamma$ is a cone point if and only if 
\begin{align}\label{e:bishop-jones}
    \int_0^1 \beta_{\Gamma}^{\infty} (x,r)^2 \, \frac{dr}{r} < + \infty.
\end{align}
\end{citetheorem}
\noindent
Here $\beta_\Gamma^\infty(x,r):= \inf_{L}\sup_{y \in B(x,r)\cap \Gamma} \frac{\dist(y, L)}{r}$, where the infimum is taken over all lines in the plane. These are the so-called Jones $\beta$-numbers; they, and their variants, have been widely employed in geometric measure theory and analysis over the past thirty years. 
Our main result is a full $d$-dimensional analogue of Theorem \ref{t:bishop-jones}. In fact, even with $d=1$ and $n=2$, our theorem is new, since it applies to any set, and not only Jordan curves.
\begin{maintheorem}\label{t:main}
Let $E \subset \R^n$ and $1 \leq d < n$. Then, up to a set of $\hd$ measure zero, $x \in E$ is a $d$-cone point if and only if 
\begin{align}\label{e:main-eq}
    \int_0^1 \beta_E^{d,p}(x,r)^2 \frac{dr}{r} < + \infty.
\end{align}
\noindent
Here, $p \in [1,p(d)],$ where
\begin{align}\label{e:pd}
    p(d) := \begin{cases}
    \infty & \mbox{ if } d=1 \mbox{ or } d=2. \\
    \frac{2d}{d-2} & \mbox{ otherwise.}
    \end{cases}
\end{align}
\end{maintheorem}
\begin{remark}[Finiteness assumption]
 Classical characterisations of rectifiability such as Theorem \ref{t:federer} (or Besicovitch projection theorem, or Preiss' results on tangent measures) usually assume $\hd(E)<+\infty$ (at least locally); this assumption is needed for the limit \eqref{e:federer-char} to even make sense. 
 This is in contrast to Theorems \ref{t:bishop-jones} and \ref{t:main}, which instead quantify how $E$ should behave near a $d$-cone point $x$, regardless of the dimension and measure of $E$. 
\end{remark}

\begin{remark}[$\beta$'s coefficients]\label{r:betas}
The coefficients appearing in \eqref{e:main-eq} are an `$L^p$' variant of the coefficients  $\beta_E^\infty$, introduced in \cite{hyde2020TST}. They are defined as
\begin{align}\label{e:beta-matt}
    \beta_E^{d,p}(x,r) := \inf_{V} \left( \frac{1}{r^d} \int_{B(x,r)\cap E} \left(\frac{\dist(y, V)}{r}\right)^p \, d\dM^d_\infty(y) \right)^{\frac{1}{p}}.
\end{align}
Here the infimum is over all $d$-dimensional affine planes in $\R^n$. The integration is with respect to $\dM_\infty^d$, a modified Hausdorff content (see Section \ref{s:prelim}). 
It is well known that the coefficients $\beta_E^\infty$ are not suitable when $d>2$\footnote{Fang \cite{fang1990} constructed a Lipschitz graph where \eqref{e:bishop-jones} is infinite on a positive measure set.}, thus to prove a theorem like Theorem \ref{t:main}, one is forced to modify them. We chose to use the coefficients firstly introduced by the first named author in \cite{hyde2020TST}. They are better in the current situation than those used in e.g. \cite{david-semmes91, azzam2015characterization, edelen2016quantitative} or in \cite{azzam2018analyst, villa2020tangent} because \textit{they do not need density assumptions on $E$ to be meaningful}. Instead, the modified Hausdorff content $\mathcal{M}_\infty^d$ is defined in such a way so that \textit{every} ball centred on $E$ is given ``large mass" (even, for example, balls centred on singleton points). By integrating with respect to $\mathcal{M}_\infty^d$, one can detect the geometry in something as sparse as a discrete set. See Definition \ref{IntroDef} for the precise definition.  
\end{remark}

\begin{remark}[Comparisons to literature]\label{r:literature}
There are several results that classify rectifiability in terms of geometric square functions. Most of them, however, are closer to Theorem \ref{t:federer} than to Theorem \ref{t:main}, since they assume the existence of a (locally) finite measure.
A particularly similar-looking result which falls in this class is the following theorem. 
\begin{citetheorem}[\cite{azzam2015characterization, tolsa2015characterization, edelen2016quantitative}] \label{t:azzam-tolsa}
Let $\mu$ be a finite Radon measure on $\R^n$ and suppose that $0< \theta^*(\mu, x)$ and $\theta_*(\mu,x)< \infty$ for $\mu$-almost all $x \in\R^n$. Then $\mu$ is $d$-rectifiable if and only if
\begin{align}\label{e:azzam-tolsa}
    \int_0^1 \beta_{\mu, 2}^d(x,r)^2 \, \frac{dr}{r} < \infty \mbox{ for } \mu-a.e. x \in \R^n.
\end{align}
\end{citetheorem}
\noindent
Here 
\begin{align} \label{e:beta-david-semmes}
    \beta_{\mu,p}^d(x,r) :=\inf_{V} \left( \frac{1}{r^d} \int_{B(x,r)} \left( \frac{\dist(y, V)}{r} \right)^p \,d\mu(y) \right)^{\frac{1}{p}},
\end{align}
where $\mu$ is a Radon measure on $\R^n$. This $L^p$ variant of Jones' $\beta^E_\infty$ was introduced by David and Semmes \cite{david-semmes91} in their work on uniform rectifiability and singular integrals. It was first proven in \cite{azzam2015characterization, tolsa2015characterization} with the stronger assumption $0<\theta^*(\mu,x)<\infty$ $\mu$-a.e. $x \in \R^n$, and then improved to its current form in \cite{edelen2016quantitative, tolsa2019rectifiability}.
There are several other pointwise results that classify rectifiability of Radon measures via geometric square functions. For example, see \cite{dkabrowski2020cones} for a nice quantification of Theorem \ref{t:federer} in terms of `conical energy'. A closely related result is that of \cite{badger2020radon}, where the authors study \textit{Federer rectifiability} in terms of conical defect. See also \cite{lerman2003quantifying, badger2015multiscale,  martikainen2018boundedness}.
For a more comprehensive list, we refer to the nice survey \cite{badger2019generalized}.

Perhaps the closest result to date to Theorem \ref{t:main} is a theorem by the second author in \cite{villa2020tangent}, where, however, it was assumed \textit{a priori} that \textit{$E$ is lower content $d$-regular}. The recent proof of Carleson $\ve^2$-conjecture \cite{jaye2019proof} by Jaye, Tolsa and the second named author is an example of the type of characterisation given in the Bishop-Jones theorem. 
\end{remark}

\begin{remark}[Approximate tangents and cone points]
A further difference between theorems \textit{\`{a} la} Federer (e.g. Theorem \ref{t:bishop-jones}, Theorem \ref{e:azzam-tolsa}) and theorems \textit{\`{a} la} Bishop-Jones (e.g. Theorem \ref{t:bishop-jones}, Theorem \ref{t:main}) is that the latter type characterise \textit{cone points}, rather than \textit{approximate tangents}. If we have some kind of lower regularity (be it Ahlfors-David or content), the two notions coincide (whenever the notion of approximate tangent makes sense). However, in the absence of lower regularity, cone point implies approximate tangent, but not vice versa.
\end{remark}

\begin{remark}[Range of $p$]
The $\beta$ numbers \eqref{e:beta-matt} (as \eqref{e:beta-content} below) have been introduced to study quantitative rectifiability beyond the Ahlfors regular setting (see \cite{azzam2018analyst, villa2020higher, azzam2019quantitative, hyde2020TST}). In this sense the range of $p$ is expected. Note however that Theorem \ref{t:azzam-tolsa} holds only for $p=2$. 
\end{remark}

We now introduce our second main result. One says that $E$ is $(d, k, \alpha)$-rectifiable if there are $d$-dimensional $C^{k, \alpha}$ graphs $\Gamma_i$, $i=1,2,...$ so that
\begin{align*}
    \hd\left(E \setminus \bigcup_{i} \Gamma_i\right) = 0. 
\end{align*}
Higher order rectifiability was, to the best of our knowledge, first considered in \cite{anzellotti1994k}.
Let $V \subset \R^n$ be a $d$-dimensional affine plane. For $x \in \R^n$, define the \textit{truncated paraboloid} with smoothness parameter $\alpha>0$, radius $r>0$ and aperture $\lambda>0$ by 
\begin{align*}
    X_\alpha(x,V, \lambda, r) := \{y \in \R^n \, |\, |\Pi_{V^\perp}(x-y)| < \lambda |\Pi_V(x-y)|^{1+\alpha} \} \cap B(x,r).
\end{align*}
Define $X^c_\alpha(x, V, \lambda, r) := B(x,r) \setminus X_\alpha(x, V, \lambda, r)$. For a set $E$, we say that $x \in E$ is a \textit{$(d, \alpha)$-paraboloid point} (or just paraboloid point) if there is a $d$-plane $V$ so that for all $\lambda>0$, we can find an $r>0$ so that $E\cap X^c_\alpha(x,V, \lambda, r) = \emptyset$. Define
\begin{align}\label{e:beta-as}
    \overline{\beta}_E^{d,p}(x,r):= \inf_V \left( \frac{1}{r^d} \int_{B(x,r)\cap E} \left(\frac{\dist(y, V)}{r}\right)^p \, d \hdc(y) \right)^{\frac{1}{p}},
\end{align}
where the infimum is over all affine $d$-planes. 
Our second main result is the following.
\begin{maintheorem}\label{t:main2}
Let $E \subset \R^n$, $1 \leq d < n$, $1 \leq p \leq p(d)$ and $0<\alpha<1$. Then, up to a set of $\hd$ measure zero, 
\begin{enumerate}
    \item if
\begin{align}\label{e:dini-matt-beta}
     \int_0^1 \beta_{E}^{d,p} (x,r)^2 \frac{dr}{r^{1+2\alpha}} < \infty,
\end{align}
then $x$ is an $\alpha$-paraboloid point. 

\item On the other hand, 
\begin{align}\label{e:alpha-second-direction}
     \int_0^1 \overline{\beta}_{E}^{d,p} (x,r)^2 \frac{dr}{r^{1+2\alpha}} < \infty,
\end{align}
whenever $x$ is an $\alpha$-paraboloid point
\end{enumerate}
\end{maintheorem}
Theorem \ref{t:main2} and Theorem \ref{t:delnin} (see below), imply immediately the following corollary.
\begin{corollary}
Let $E \subset \R^n$. If \eqref{e:dini-matt-beta} holds for $\hd$-almost all $x \in E$, then $E$ is $(d, 1, \alpha)$-rectifiable.
\end{corollary}
There is another corollary of Theorem \ref{t:main2}, essentially due to the fact that $\beta_E^{d,p} \approx \overline{\beta}_E^{d,p}$ if $E$ is $d$-LCR (see Remark \ref{r:beta-alpha}).
\begin{corollary}
Let $E \subset \R^n$ be $d$-LCR, $1 \leq p \leq p(d)$ and $0<\alpha<1$. Then, up to a set of $\hd$ measure zero, $x \in E$ is an $\alpha$-paraboloid point if and only if 
\begin{align*}
    \int_0^1 \overline{\beta}_{E}^{d,p} (x,r)^2 \frac{dr}{r^{1+2\alpha}} < \infty.
\end{align*}

\end{corollary}

\begin{remark}[$\beta$ coefficients]\label{r:beta-alpha}
Note that the integration in \eqref{e:beta-as} is with respect to Hausdorff content (rather than $\dM_\infty^d$ as in Theorem \ref{t:main}). The numbers $\overline{\beta}_E^{d,p}$ were introduced by Azzam and Schul \cite{azzam2018analyst} to give a first example of the Analyst's Travelling Salesman Theorem in higher dimensions. They are closely related to $\beta_{E}^{d,p}$ (as in \eqref{e:beta-matt}) but they really become geometrically meaningful when $E$ is lower content $d$-regular. Recall that $E$ is said to be \textit{lower content} $(d,c)$-\textit{regular} ($d$-LCR) when 
\begin{align}
    \hdc(B(x,r)\cap E) \geq c r^d \mbox{ for } x \in E \mbox{ and } r>0.
\end{align}
Indeed, if $\hdc(B(x,r)\cap E)$ is very small, $\overline{\beta}_E^{d,p}(x,r)$ will be very small even if the set is not actually concentrated around a plane. 
\end{remark}

\begin{remark}
We could not prove Theorem \ref{t:main2}(2) using $\beta^{d,p}_E$ instead of $\overline{\beta}^{d,p}_E$ because certain good properties of $\overline{\beta}_E^{d,p}$ do not hold for $\beta_E^{d,p}$. We will say more on this in the proof. 
\end{remark}

\begin{remark}[Comparison to literature]
Higher order rectifiability has been studied in, for example, \cite{anzellotti1994k, delladio2004result, delladio2008sufficient, santilli2017rectifiability, ghinassi2017sufficient, del2019geometric}.
In analogy to Theorem \ref{t:federer}, 
Del Nin and Obinna Idu \cite{del2019geometric} proved Theorem \ref{t:delnin} below. In fact, if $c(d) = 0$, then the first half of Theorem \ref{t:delnin} is a special case of \cite[Theorem 3.23]{santilli2017rectifiability}, and the converse is a special case of \cite[Theorem 5.4]{santilli2017rectifiability}.
\begin{citetheorem}[\cite{del2019geometric}, Theorem 1.1 and Prop. 1.2]\label{t:delnin}
Let $d<n$ an integer, $0<\alpha\leq 1$, $E \subset \R^n$ with $\hd(E)<\infty$. If for $\hd$-a.e. $x \in E$, there exists a $d$-plane $V_x$ and $\lambda>0$ so that
\begin{align*}
    \limsup_{r \to 0} \frac{\hd(X_\alpha^c(x, V_x, \lambda, r)\cap E)}{r^d} < c(d),
\end{align*}
then $E$ is $C^{1,\alpha}$ rectifiable. Conversely, if $E$ is $C^{1,\alpha}$ rectifiable, then 
\begin{align*}
    \lim_{r \to 0} \frac{\hd(X_\alpha^c(x, V_x, \lambda, r)\cap E)}{r^d} = 0.
\end{align*}
\end{citetheorem}
\noindent
One direction of the Azzam-Tolsa Theorem \ref{t:azzam-tolsa} in this context was shown to hold by Ghinassi \cite{ghinassi2017sufficient}. It was later improved in \cite{del2019geometric} (Corollary 1.6).
\begin{citetheorem}[\cite{ghinassi2017sufficient}, Theorems I, II] \label{t:ghinassi}
Let $E \subset \R^n$, with $0<\theta^*_d(E, x)<\infty$ and $0<\alpha<1$. Put $p=2$ or $p=\infty$. Suppose that for $\hd$-a.e. $x \in E$, 
\begin{align*}
   \int_0^1 \beta_{E, p}^d (x,r)^2 \frac{dr}{r^{1+2\alpha}} < \infty. 
\end{align*}
Then $E$ is $C^{1,\alpha}$ rectifiable. Here $\beta_{E, p}^d$ is as in \eqref{e:beta-david-semmes}.
\end{citetheorem}
While Theorem \ref{t:main2} is a characterisation, it does not provide a converse to Theorem \ref{t:ghinassi}. 
 See \cite{chousionis2019boundedness, Idu2021C} for results on $C^{1,\alpha}$ rectifiability in Heisenberg groups.
\end{remark}

\subsection{Sketch of the proof}
Theorem \ref{t:main2}(1) and hence the ``if'' direction of Theorem \ref{t:main} is proved in Section \ref{s:dini-tangent}. We split the Christ-David cubes intersecting the points satisfying \eqref{e:dini-matt-beta} (or \eqref{e:main-eq}) into stopping-time regions for which we have good control on a Bishop-Jones type square function and each cube $Q \in S$ has sufficiently many separated points from $E$. Because we can control angles between best approximating planes when there are sufficiently many separates points (Lemma \ref{l:angpre}), we can apply the parameterisation theorems of Ghinassi (Theorem \ref{l:azzam}) to  produce a $C^{1,\alpha}$-surface $\Sigma_S$ which well approximates $E$ inside $S.$  For $\mathcal{H}^d$-almost every $x$ satisfying \eqref{e:dini-matt-beta} (or \eqref{e:main-eq}), we can find a stopping-time region $S$ so that $x \in \Sigma_S$ and so that $x$ is a paraboloid point for $\Sigma_S.$ The final part of the proof is to show that these are actually paraboloid points for $E$.

That being a cone point (resp. a paraboloid point) implies \eqref{e:main-eq} (resp. \eqref{e:alpha-second-direction}) is proved in Section \ref{s:tan-beta}. We consider subsets of cone/paraboloid points that can be covered by Lipschitz or $C^{1,\alpha}$ graphs. Then, in the case of Theorem \ref{t:main} we construct a $d$-LCR set $E_0$ containing $E$ and which have the same cone points. This ultimately permit us to compute the $\beta$ coefficients. This strategy is somewhat inspired by Bishop and Jones original one, even if the final technique is very different. The methods for Theorem \ref{t:main2} are similar, with some technical differences.

\subsection*{Acknowledgments}
We thank Jonas Azzam for several useful conversations. We also thank Giacomo Del Nin and Mario Santilli for some relevant comments on a first draft of this paper which improved the exposition.

\section{Preliminaries}\label{s:prelim}

\subsection{Notation} \label{sec:notation}
We gather here some notation and some results which will be used later on.
We write $a \lesssim b$ if there exists a constant $C$ such that $a \leq Cb$. By $a \sim b$ we mean $a \lesssim b \lesssim a$.
In general, we will use $n\in \N$ to denote the dimension of the ambient space $\R^n$, while we will use $d \in \N$, with $d\leq n-1$, to denote the dimension of a subset $E \subset \R^n$.

For two subsets $A,B \subset \R^n$, we let
$
    \dist(A,B) := \inf_{a\in A, b \in B} |a-b|.
$
For a point $x \in \R^n$ and a subset $A \subset \R^n$, 
$
    \dist(x, A):= \dist(\{x\}, A)= \inf_{a\in A} |x-a|.
$
We write 
$
    B(x, r) := \{y \in \R^n \, |\,|x-y|<r\},
$
and, for $\lambda >0$,
$
    \lambda B(x,r):= B(x, \lambda r).
$
At times, we may write $\B$ to denote $B(0,1)$. When necessary we write $B_n(x,r)$ to distinguish a ball in $\R^n$ from one in $\R^d$, which we may denote by $B_d(x, r)$.

\subsection{Christ-David and Whitney cubes.}

In this section we briefly recall two `cubes'-tools which will be needed later on. 
\subsubsection{Christ-David cubes} David in \cite{david-wavelets} and then Christ in \cite{christ1990b}, introduced a construction which allows for a partition (up to a set of measure zero) of any space of homogeneous type into open subsets which resemble the behaviour of dyadic cubes in $\R^n$. The formulation below is due to Hyt\"{o}nen and Martikainen \cite{hytonen2012non} and forms an exact partition of any doubling metric space $X.$

\begin{comment}
Recall that a space of homogeneous type is a set $X$ equipped with a quasi-metric $p$ and a Borel measure $\mu$ so that 
\begin{itemize}
\item The balls associated to $p$ are open.
\item $\mu(B(x,r)) < \infty$ for all $x \in X$, $r>0$.
\item $\mu$ satisfies the doubling condition with respect to these balls, 
that is, there exists a constant $A$, independent of $x$ and $r$ so that
\begin{align*}
    \mu(B(x, 2r)) \leq A \mu(B(x,r)).
\end{align*}
\end{itemize}
\end{comment}

\begin{theorem}[Christ-David cubes] \label{theorem:christ}
Let $X$ be a doubling metric space and $X_k$ be a sequence of maximal $\rho^k$-separated nets, where $\rho = 1/1000$ and let $c_0 = 1/500.$ Then, for each $k \in \Z$, there is a collection $\mathscr{D}_k$ of cubes such that the following hold.
\begin{enumerate}
\item For each $k \in \Z, \ X = \bigcup_{Q \in \mathscr{D}_k}Q.$
\item If $Q_1,Q_2 \in \mathscr{D} = \bigcup_{k}\mathscr{D}_k$ and $Q_1 \cap Q_2 \not= \emptyset,$ then $Q_1 \subseteq Q_2$ or $Q_2 \subseteq Q_1.$ 
\item For $Q \in \mathscr{D},$ let $k(Q)$ be the unique integer so that $Q \in \mathscr{D}_k$ and set $\ell(Q) = 5\rho^k.$ Then there is $x_Q \in X_k$ such that
\begin{align*}
B(x_Q,c_0\ell(Q)) \subseteq Q \subseteq B(x_Q , \ell(Q)). 
\end{align*}
\end{enumerate}
\end{theorem}

The $k^{th}$-\textit{generation children} of $Q \in \mathcal{D}$, denoted by $\text{Child}_k(Q),$ are the cubes $R \subseteq Q$ so that $\ell(R) = \rho^k\ell(Q).$ We also need the notion of a stopping-time region. 

\begin{definition}\label{d:ST}
A collection of cubes $S \subseteq \mathcal{D}$ is called a \textit{stopping-time region} if the following hold.
\begin{enumerate}
\item There is a cube $Q(S) \in S$ such that $Q(S)$ contains all cubes in $S$. 
\item If $Q \in S$ and $Q \subseteq R \subseteq Q(S),$ then $R \in S$.
\item If $Q \in S$, then all siblings of $Q$ are also in $S$. 
\end{enumerate}
\end{definition}

\bigskip

\subsubsection{Whitney cubes}
We follow \cite{stein1993harmonic}, Lemma 2 of Chapter 1. Let $F$ be a closed non-empty subset of $\R^n$ endowed with the standard euclidean distance. Then there exists a collection $\{S_k\}$ of closed cubes with disjoint interiors which covers $F^c$ (the complement of $F$) and whose side lengths are comparable to their distance to $F$, that is, there exists a constant $A$ so that
\begin{align*}
    A^{-1} \dist(S_k, F) \leq \ell(S_k) \leq A \dist(S_k, F).
\end{align*}

\subsection{Parameterisation Theorems}

Two main tools that we shall need are the parameterisation results of David and Toro \cite{david2012reifenberg} and Ghinassi \cite{ghinassi2017sufficient}. To state these results, we need to introduce some more definitions. 

\begin{definition}
Set the normalised local Hausdorff distance between two sets $E, F$ to be given by
\begin{align}
    d_{x,r}(E,F) := \frac{1}{r} \max \left\{ \sup_{y \in E\cap B(x,r)} \dist(y, F) \,; \, \sup_{y \in F \cap B(x,r)} \dist(y, E) \right\}. \label{def:normalisedhausdorff}
\end{align}
\end{definition}
\noindent
Moreover, for $x \in \R^n$, $r>0$, we set
\begin{align*}
    \vartheta^d_E(B(x,r)) := \inf \{d_{B(x,r)}(E, L) \,|\, L \mbox{ is a } d \mbox{-dimensional plane in } \R^n\}.
\end{align*}

\begin{definition}
A subset $E \subset \R^n$ is said to be $(d, \epsilon)$-Reifenberg flat (or simply $\epsilon$-Reifenberg flat if the dimension is understood) if 
\begin{align*}
    \vartheta_E^d(B(x,r)) < \epsilon \mbox{ for all } x \in E \mbox{ and } r>0.
\end{align*}
\end{definition}

We will implement the parameterisation results \cite{david2012reifenberg, ghinassi2017sufficient} via the following reformulation in terms of Christ-David cubes. This formulation is proven in the Appendix of \cite{azzam2019harmonic} for the bi-Lipschitz case and the proof is the same (modulo replacing the improved estimates in \cite{ghinassi2017sufficient}) for the $C^{1,\alpha}$ case.

\begin{theorem}[\cite{david2012reifenberg}, {\cite{ghinassi2017sufficient}}]\label{l:azzam}
Let $C_1,C_2 > 1.$ Let $F \subseteq \R^n$ and let $\mathcal{D}$ be the Christ cubes for $F.$ For $Q,R \in \mathcal{D}$, declare $R \sim Q$ if $C_2^{-1}\ell(Q) \leq \ell(R) \leq C_2\ell(Q)$ and $\dist(Q,R) \leq C_2\min\{\ell(Q),\ell(R)\}.$ Then, for $C_1,C_2$ large enough and $\ve \ll C_1^{-1},C_2^{-1},$ the following holds. Let $S$ be a stopping-time region with top cube $Q(S)$ so that $x_{Q(S)} = 0$ and for all $Q \in S$, there is a $d$-plane $L_Q$ such that
\begin{align}\label{e:azzam1}
\dist(x_Q,L_Q) < \ve\ell(Q).
\end{align}
Moreover, if 
\begin{align}\label{e:azzam2}
\ve(Q) = \max_{\substack{R \in S \\ R \sim Q}}  d_{C_1B_Q}(L_Q,L_R)
\end{align}
assume
\begin{align}\label{e:azzam3}
\sum_{Q \subseteq R \subseteq Q(S)} \ve(R)^2 < \ve^2. 
\end{align}
Then for $\ve >0$ small enough, there is a map $g: L_{Q(S)} \rightarrow \R^n$ (invertible on its image) which is $(1+C\ve)$-bi-Lipschitz and so that surface $\Sigma_S = g(L_{Q(S)})$ is $C\ve$-Reifenberg flat. In addition, if $0 < \alpha < 1$ and 
\begin{align}\label{e:azzam3.5}
\sum_{Q \subseteq R \subseteq Q(S)} \frac{\ve(R)^2}{\ell(R)^{2\alpha}} < \ve^2 
\end{align}
then $g$ and $g^{-1}$ are $C^{1,\alpha}.$
Finally, if 
\begin{align}
\sup_{z \in F \cap 2C_1B_Q} \dist(z,L_Q) \leq \ve \ell(Q) \ \text{for all} \ Q \in S
\end{align}
then 
\begin{align}\label{e:azzam4}
\sup_{z \in F \cap C_1B_Q} \dist(z,\Sigma_S) \lec \ve \ell(Q) \ \text{for all} \ Q \in S.
\end{align}
\end{theorem}

\begin{remark}
The surface $\Sigma_S$ satisfies several other properties that are not needed here. See \cite{azzam2019harmonic} for the complete statement. 
\end{remark}

\subsection{Choquet integration and $\beta$-numbers}
For $1 \leq p<\infty$ and $A \subset \R^n$ Borel, we define the $p$-Choquet integral as 
\begin{align*}
    \int_A f(x)^p\, d \hdc(x) := \int_0^\infty \hdc(\{x \in A\, |\, f(x)>t\}) \, t^{p-1}\, dt.
\end{align*}
 We refer the reader to \cite{mattila} for more detail on Hausdorff measures and content and to Section 2 and the Appendix of \cite{azzam2018analyst} for more details on Choquet integration. We highlight the following property of Choquet integration; it satisfies a Jensen's inequality.

\begin{lemma}\label{l:jensen}
Let $E \subseteq \R^n$ be either compact or bounded and open so that $\hd(E) >0,$ and let $f \geq 0$ be continuous on $E$. Then for $1 < p \leq \infty,$
\[ \frac{1}{\hd_\infty(E)} \int_E f \, d\hd_\infty \lesssim_n \left( \frac{1}{\hd_\infty(E)} \int_E f^p \, d\hd_\infty \right)^\frac{1}{p} \] 
\end{lemma}

We recall the various $\beta$-numbers that we will use and prove/state some of their properties. 

\begin{definition}[Jones]
Let $E \subseteq \R^n$ and $B$ a ball. Define
\begin{align*}
\beta_{E,\infty}^d(B) = \frac{1}{r_B}\inf_L \sup\{\text{dist}(y,L) : y \in E \cap B\}
\end{align*}
where $L$ ranges over $d$-planes in $\R^n.$ 
\end{definition}

\begin{definition}[Azzam-Schul]
Let $1 \leq p < \infty,$ $E \subseteq \R^n$ and $B$ a ball. For a $d$-dimensional plane $L$ define
\begin{align} \label{e:beta-content} \overline{\beta}^{d,p}_E(B,L) = \left( \frac{1}{r_B^d} \int_{E \cap B} \left( \frac{\dist(y,L)}{r_B} \right)^p \, d\mathcal{H}^d_\infty(y) \right)^\frac{1}{p}.
\end{align}
\end{definition}

Let us be more precise about the $\beta$-number introduced in \cite{hyde2020TST}. We first define the modified Hausdorff content, $\mathcal{M}^d_\infty$.

\begin{definition}\label{IntroDef}
Let $E \subseteq \R^n$, $B$ a ball and $0< c_1 \leq c_2 < \infty$ be constants to be fixed later. We say a collection of balls $\mathscr{B}$ which covers $E \cap B$ is \textit{good} if 
\begin{align}\label{Size}
\sup_{B' \in \mathscr{B}} r_{B'} < \infty
\end{align}
and for all $x \in E \cap B$ and $0 <r <r_B,$ we have
\begin{align}\label{LRi}
\sum_{\substack{B' \in \mathscr{B} \\ B' \cap B(x,r) \cap E \cap B \not= \emptyset}} r_{B'}^d \geq c_1 r^d
\end{align}
and
\begin{align}\label{URi}
\sum_{\substack{B' \in \mathscr{B} \\ B' \cap B(x,r) \cap E \cap B \not=\emptyset \\ r_{B'} \leq r}} r_{B'}^d \leq c_2 r^d.
\end{align}
Then, for $A \subseteq E \cap B,$ define 
\begin{align*}
\mathcal{M}^d_\infty(A) = \mathcal{M}^d_\infty(A,E,B) = \inf\left\{ \sum_{\substack{B^\prime \in \mathscr{B}\\ B^\prime \cap A \not= \emptyset}} r_{B^\prime}^d : \mathscr{B} \ \text{is good for} \ E \cap B \right\}
\end{align*}
\end{definition}
\noindent

\begin{definition}
Let $1 \leq p <\infty$, $E \subseteq \R^n,$ $B$ a ball centred on $E$ and $L$ a $d$-plane. Let $\mathcal{M}^d_\infty$ be the content from Definition \ref{IntroDef} for $E \cap B.$ Then, define
\begin{align*}
\beta_E^{d,p}(B,L)^p &= \frac{1}{r_B^d} \int \left( \frac{\text{dist}(x,L)}{r_B}\right)^p \, d\mathcal{M}^d_\infty \\
&= \frac{1}{r_B^d} \int_0^1 \mathcal{M}^d_\infty(\{x \in E \cap B : \text{dist}(x,L) >tr_B \}) t^{p-1} \, dt,
\end{align*}
and 
\begin{align}\label{e:beta-matt}
\beta_E^{d,p}(B) = \inf \{ \beta_E^{d,p}(B,L) : L \ \text{is a $d$-plane} \}. 
\end{align}
\end{definition}

The properties of $\mathcal{M}^d_\infty$ are not so important for us here. We refer the reader to \cite{hyde2020TST} for more details, where the same quantity is denoted by $\mathscr{H}^{d,E}_{B,\infty}$. The following lemmas are to be found in \cite{hyde2020TST}, Section 2. They will be used later on in the various proofs. 

\iffalse
\begin{proof}
For any subset $A \subset \R^n$, recall that $\hdc(A) \leq \diam(A)^d$. Then
\begin{align*}
    \betae{1}(x,t) & = \inf_L \frac{1}{t^d} \int_{E \cap B(x, t)} \frac{\dist(y, L)}{t} \, d\hdc(y) \\
    & \leq \inf_L \frac{1}{t^d} \sup_{y \in B(x,t) \cap E} \frac{\dist(y, L)}{t} \hdc(E \cap B(x,t))\\
    & \leq 2^d \inf_L \sup_{y \in B(x,t)\cap E} \frac{\dist(y, L)}{t}\\
    & = 2^d \beta_{E, \infty}^d.
\end{align*}
\end{proof}
\fi

\begin{lemma}[{\cite[Lemma 2.23]{hyde2020TST}}]\label{lemma:betaincreasep}
Let $1 \leq p < \infty, \ E \subseteq \R^n$ and $B$ a ball centred on $E$. Then
\begin{align}
\beta_E^{d,1}(B) \lesssim_{d,p} \beta_E^{d,p}(B)
\end{align}
\end{lemma}

\begin{lemma}[{\cite[Lemma 2.24]{hyde2020TST}}] \label{lemma:monotonicity}
Let $1 \leq p < \infty$ and $E \subseteq \R^n.$ Then, for all balls $B^\prime \subseteq B$ centred on $E,$
\begin{align}
\beta^{d,p}_E(B^\prime) \leq \left(\frac{r_B}{r_{B^\prime}} \right)^{1+\frac{d}{p}} \beta_E^{d,p}(B). 
\end{align}
\end{lemma}

\begin{lemma}[{\cite[Corollary 2.26]{hyde2020TST}}] \label{l:comparable_beta}
Suppose $c \geq c_1$ (the constant appearing in Definition \ref{IntroDef}) and $E \subseteq \R^n$ is $(d,c)$-lower content regular. For any ball $B$ centred on $E$ and $d$-plane $L,$ we have
\begin{align}
\overline\beta^{d,p}_E(B,L) \leq \beta_E^{d,p}(B,L) \lesssim \overline\beta^{d,p}_E(2B,L).
\end{align}
\end{lemma}

\begin{lemma} \label{lemma:betap_betainfty}
Let $\alpha \in [0,1).$ Assume $E \subseteq \R^n$ and $B$ is a ball centred on $E$. Then 
\begin{align}\label{eq:betap_betainfty}
r^{-\alpha}\beta_{E,\infty}^d\left( \frac{1}{2}B\right) \lec \left( r^{-\alpha}\beta_E^{d,1}(B)\right)^\frac{1}{d+1}.
\end{align}
\end{lemma}

\begin{proof}
Let $L$ be a $d$-plane such that 
\[ \beta^{d,1}_E(B,L) \leq 2 \beta^{d,1}_E(B). \]
Let $y \in E \cap \frac{1}{2}B$ be the point furthest from $L$ and set $\tau = \dist(y,L)/r_B^{1+\alpha}.$ Notice that $B(y,\tau/2) \subseteq \{x \in E \cap B : \dist(x,L) > tr_B^{1+\alpha}\}$ for all $0 < t < \tau/2.$ Let $\mathcal{M}_\infty^d$ denote the content for $E \cap B.$ It is immediate from Definition \ref{IntroDef} that 
\[ \mathcal{M}_\infty^d(B(y,\tau/2)) \gtrsim \tau^d \] 
and 
\[ \mathcal{M}_\infty^d(B(y,\tau/2)) \leq \mathcal{M}_\infty^d(\{x \in E \cap B : \dist(x,L) > tr_B^{1+\alpha}\}). \] 
Hence,
\begin{align*}
1 &\lec \frac{\mathcal{M}_\infty^d(B(y,\tau/2))}{\tau^d} \\
&\lec \frac{1}{\tau^{d+1}} \int_0^{\tau/2} \mathcal{M}_\infty^d(\{x \in E \cap B : \dist(x,L) > tr_B^{1+\alpha} \}) \, dt \\
&\leq \frac{r_B^{-\alpha}}{\tau^{d+1}} \int_0^{1}\mathcal{M}_\infty^d(\{x \in E \cap B : \dist(x,L) > tr_B \}) \, dt \\
&\lec \frac{r_B^{-\alpha}}{\tau^{d+1}} \beta^{d,1}_E(B).
\end{align*}
Then 
\begin{align}
r_B^{-\alpha}\beta^d_{E,\infty}\left(\frac{1}{2}B\right) \lec \tau \lec \left(r_B^{-\alpha}\beta^{d,1}_E(B)\right)^\frac{1}{d+1}.
\end{align}
\end{proof}

The following lemmas let us relate $\beta$-numbers defined over different sets.

\begin{lemma}[{\cite{azzam2018analyst}, Lemma 2.21}] \label{lemma:azzamschul}
Let $1 \leq p < \infty$ and $E_1, E_2 \subset \R^n$. Let $x \in E_1$ and fix $r>0$. Take some $y \in E_2$ so that $B(x,r) \subset B(y, 2r)$. Assume that $E_1, E_2$ are both lower content $c$-regular. Then
\begin{align*}
    \overline{\beta}_{E_1}^{d,p} (x,r) \lesssim_c \overline{\beta}_{E_2}^{d,p} (y, 2r) + \left( \frac{1}{r^d} \int_{E_1 \cap B(x, 2t)} \left(\frac{\dist (y, E_2)}{r} \right)^p \, d \hdc(y)\right)^{\frac{1}{p}}.
\end{align*}
\end{lemma}

\begin{lemma}[{\cite[Lemma 2.25]{hyde2020TST}}]\label{l:error2}
Let $F \subset \R^n$ be a lower content $(d, c)$-regular set. Suppose that $E \subset F$. Let $B$ be a ball and $L$ a $d$-dimensional plane. Then we have
\begin{align*}
    \beta_{E}^{d,p} (B, L) \lesssim_{c, d, p} \beta_{F}^{d, p} (2B, L) +\left(\frac{1}{r_B^d} \int_{F \cap 2B} \left(\frac{\dist(x, E)}{r_B}\right)^{p} \, d \hd_\infty(x) \right)^{\frac{1}{p}}.
\end{align*}
\end{lemma}

\iffalse
\begin{proof}
Let $L(x,t)$ (resp. $L(y, s)$) the infimizer $d$-plane of $\betae{1}(x,t)$ (resp. of $\betae{1}(y,s)$. Then
\begin{align}
    \betae{1}(y,t) & = \frac{1}{s^d} \int_{E \cap B(y,s)} \frac{\dist(z, L(y, s))}{s} \, d\hdc(y)\nonumber \\
    & \leq \frac{1}{s^d} \int_{E \cap B(x,t)} \frac{\dist(z, L(x, t))}{s} \, d\hdc(y) \nonumber \\
    & = \frac{1}{s^d} \int_0^1 \hdc(\{z \in B(x,t) \cap E \,|\, \dist(z, L(x,t)) > s \tau \}\, d \tau \label{eq:monotonicity1}
\end{align}
We make the change of variable $\tau = \frac{t}{s} \tilde \tau$. Then, continuing the chain of inequalities,
\begin{align*} 
\eqref{eq:monotonicity1} & \leq \frac{1}{s^d} \int_0^1 \hdc(\{ y z \in B(x,t) \cap E \, |\, \dist(z, L(x,t)) > t \tilde\tau \}) \, \frac{t}{s} d \tilde \tau \\
& = \left(\frac{t}{s}\right)^{d+1} \betae{1}(x,t).
\end{align*}
\end{proof}
\fi

%\input{03b_Citations}
%\input{07_mods}
\section{Fast decay of $\beta$ coefficients implies existence of paraboloid points} \label{s:dini-tangent}

In this section we will show the first part of Theorem \ref{t:main2}. In particular, we show that up to a set of zero $\hd$-measure, if 
\begin{align*}
    \int_0^{\diam(E)} \frac{\beta_{E}^{d,p}(x,r)^2}{r^{2\alpha}} \, \frac{dr}{r} < + \infty,
\end{align*}
then $x$ is an $\alpha$-paraboloid point (or cone point when $\alpha=0$). We actually show that $\mathcal{H}^d$-almost every $x$ satisfying the above is an $\alpha$-\textit{tangent} point, meaning  
\begin{align}\label{e:alpha-tangent}
    \lim_{r \to 0} \sup_{y \in B(x,r) \cap E} \frac{\dist(y, L)}{r^{1+\alpha}} = 0.
\end{align}
However, we have the following: 

\begin{lemma}\label{l:tangent-cone}
Let $\alpha \in [0, 1)$. Let $x \in E$ be a $\alpha$-tangent point. Then $x$ is an $\alpha$-paraboloid point \textup{(}or cone point when $\alpha=0$\textup{)}. 
\end{lemma}
\begin{proof}
Let $L_x= L+x$ be the $d$-plane in the definition of $\alpha$-tangent point. 
If the lemma did not hold, for each $d$-dimensional plane $V \in G(d, n)$ we could find a $\theta \in (0,\pi/2)$ such that for all $r>0$, there existed a $y \in B(x,r) \cap E \setminus X_\alpha(x, V, \theta)$; such a $y$ would have $\dist(x-y, V) \geq \sin(\theta)|x-y|^{1+\alpha}$. One can easily see that this applied to $L$ implies that $\limsup_{r \downarrow 0} \sup_{y \in B(x,r)\cap E} \dist(y,L_x)/r \geq \sin(\theta)$. 
\end{proof}

Then, the first part of Theorem \ref{t:main2} will follow from the proposition below (see Corollary \ref{corol:final2}). 

\begin{proposition} \label{theorem:summabletangent}
Let $E \subset \R^{n}$ be closed; assume moreover that $E \subset B(0,1)$. For $k \in \Z$ set $r_k = 10^{-k}$. Let $1 \leq p < \infty$. Then $\hd$-almost every point $x \in E$ such that
\begin{align*}
\sum_{k \in \N} \frac{\beta_E^{p, d} (x, r_k)^2}{r_k^{2\alpha}} < + \infty,
\end{align*}
is an $\alpha$-tangent point of $E$.
\end{proposition}

\begin{remark}
Because of Lemma \ref{lemma:betaincreasep}, it is enough to prove Proposition \ref{theorem:summabletangent} for $p=1$. 
We may assume that the set
\begin{align*}
  G = \left\{ x \in E\, |\, \sum_{k \in \N} \frac{\betae{1} (x, r_k)^2}{r^{2\alpha}} < + \infty \right\}
\end{align*}
has non-zero $\hd$ measure, for otherwise there is nothing to prove. 
\end{remark}

\begin{lemma} \label{lemma:reduction1}
It suffices to show that if $A \subseteq G$ is compact and $0 < \hd(A) < \infty$, then $\hd$-almost every $x \in A$ is a tangent point of $E$.
\end{lemma}
\begin{proof}
Suppose that for any $A \subseteq G$ with $0 < \hd(A) < \infty$ we have that $\hd$-almost every $x \in A$ is a tangent point of $E$. We claim, for any $A' \subseteq G$ (with no assumption on the measure), that $\hd$-almost every $x \in A'$ is a tangent point of $E$.

We argue by contradiction. Suppose that there exists a subset $\Omega \subset A'$ with $0 < \hd(\Omega)$ and so that no $x \in \Omega$ is a tangent point of $E$. By Theorem 8.13 in \cite{mattila}, there exists a compact subset $K \subset \Omega$ with $0< \hd(K) < \infty$ and so that no $x \in K$ is a tangent point of $E$. This leads to a contradiction and the lemma follows.
\end{proof}

\begin{remark} \label{remark:e0}
Let now $A \subseteq G$ be an arbitrary compact subset such that $0< \hd(A) < \infty$. Again, by Theorem 8.13 in \cite{mattila}, such a set exists. 
Let $\epsilon >0 $ be given. For $\ell \in \N$, define
\begin{align*}
    E_\ell:= \Bigg\{ x \in A\, |\, \sum_{k \geq \ell}^\infty \frac{\betae{1}(x, r_k)^2}{r^{2\alpha}} < \epsilon^2\Bigg\}.
\end{align*}
For an arbitrary $\ell \in \N$ with $\hd(E_\ell) >0$, we now fix 
\begin{align} \label{eq:e_0}
    E_0 \subset E_\ell \subset A 
\end{align}
to be a compact subset of positive and finite $\hd$-measure. If we prove Proposition \ref{theorem:summabletangent} for $E_0$, then the statement in its full generality follows by first saturating $E_\ell$ with compact subsets, and then by taking the countable union $\cup_\ell E_\ell = A$. Moreover, by applying a dilation and a translation, we may assume that 
\begin{align} \label{eq:assumptioscaleplace}
E_0 \subset B(0, 1), \mbox{ with } 0 \in E_0,
\end{align}
and that $\ell = 1$ in \eqref{eq:e_0}. It follows that
\begin{align}
    \frac{\betae{1}(x, r_k)}{r^{\alpha}} < \epsilon \mbox{ for } x \in E_0 \mbox{ and for } k \geq 1.
\end{align}
\end{remark}

Let $\alpha \in [0,1).$ From this point on, we will refer to an $\alpha$-tangent point as a tangent point. Let $\mathcal{D}$ denote the Christ cubes for $E$ from Theorem \ref{theorem:christ}. We first make precise the separation condition mentioned in the outline.

\begin{definition}\label{l:ep}
Let $ 0 <\kappa <1.$ We say a ball $B$ has $(d+1,\kappa)$-separated points if there exist points $X = \{x_0,\dots,x_d\}$ in $E \cap B$ such that, for each $i=1,\dots,d,$ we have
\begin{align}\label{e:sepdef}
\text{dist}(x_{i+1},\text{span}\{x_0,\dots,x_i\}) \geq \kappa r_{B}.
\end{align}
We say that a cube $Q \in \mathcal{D}$ has $(d+1,\kappa)$-\textit{separated points} if $C_1B_Q$ has $(d+1,\kappa)$-separated points. 
\end{definition}

\begin{remark}

Because we want to control angles between planes that well approximate $E$, we require the separation in $E$ not necessarily in $E_0.$ 
\end{remark}

\begin{lemma}[{\cite[Lemma 2.32]{hyde2020TST}}]\label{l:angpre}
Let $E \subseteq \R^n$ and let $B^\prime$ and $B$ be balls centred on $E$ with $B^\prime \subseteq B.$ Suppose further that there exists $0 <\kappa <1$ such that $B^\prime$ has $(d+1,\kappa)$-separated points. Let $L$ and $L^\prime$ be two $d$-planes. Then
\begin{align*}
d_{B^\prime}(L,L^\prime) \lec  \frac{1}{\kappa^{2d+2}}\left[\left( \frac{r_B}{r_{B^\prime}}\right)^{d+1} \beta_E^{d,1}(2B,L) + \beta_E^{d,1}(2B^\prime,L^\prime) \right].
\end{align*}
\end{lemma}

\subsection{Stopping-time regions} We now define the stopping-time regions for which we can apply Theorem \ref{l:azzam}. First, let $\mathscr{B}$ be the set of cubes $Q$ in $\mathcal{D}$ so that $E_0 \cap Q \not= \emptyset$ and $Q$ does \textit{not} have $(d+1,\kappa)$ separated points. Let $\delta > 0$ (to be chosen small enough depending on $\kappa$) and $M >1$ (to be chosen large enough). For $Q \in \mathcal{D}\setminus\mathscr{B}$ such that $Q \cap E_0 \not=\emptyset$, let $\Stop(Q)$ be the maximal collection of cubes $R$, contained in $Q$, so that $R$ has a child $R'$ such that $R' \in \mathscr{B}$ or 
\begin{align}
\sum_{R' \subseteq T \subseteq Q} \frac{\beta^{d,1}_E(M^2B_T)^2}{\ell(T)^{2\alpha}} \geq \delta^2. 
\end{align}

Then, let $\tree(Q)$ be those cubes contained in $Q$ that are not properly contained in any cube from $\Stop(Q).$ In this way, each $\tree(Q)$ defines a stopping-time region in the sense of Definition \ref{d:ST}, with top cube $Q$. It will also be convenient to define $\tree(Q)$ and $\Stop(Q)$ for cubes in $\mathscr{B}.$ In this case, we set 
\[\tree(Q) = \Stop(Q) = \{Q\}.\]
Observe, if $\tree(Q)$ is not a singleton (i.e. $\tree(Q) \not=\{Q\}$) and $R \in \tree(Q),$ then $R$ has $(d+1,\kappa)$-separated points and satisfies
\begin{align}\label{e:stop-time-condition}
\sum_{R \subseteq T \subseteq Q} \frac{\beta^{d,1}_E(M^2B_T)^2}{\ell(T)^{2\alpha}} < \delta^2. 
\end{align}

\subsection{Some collections of cubes} For $Q \in \mathcal{D}\setminus \mathscr{B}$ such that $E_0 \cap Q \not=\emptyset$, let $\Stop(Q)_{\mathscr{B}}$ be the collection of cubes $R \in \Stop(Q)$ so that there exists a child $R'$ so that $R' \in \mathscr{B}.$ For $Q \in \mathscr{B}$, let $\Stop(Q)_\mathscr{B} = \{Q \}.$ 

Let $\Next(Q)$ denote the set of cubes $R$ for which there exists $R' \in \Stop(Q)_{\mathscr{B}}$ such that $R \subseteq R'$ and $\ell(R) = \rho^{k^*}\ell(R'),$ where $k^* = k^*(\kappa)$ is the smallest integer $k$ so that
\begin{align}\label{e:k^*}
\rho^{-k} \leq \frac{c_0}{2 C_1\kappa \rho}.
\end{align}
We remark here that, as $\kappa$ becomes small, so does $\rho^{k^*}.$ So, $\Next(Q)$ is the $k^*$-generation grandchildren of cubes in $\Stop(Q)_{\mathscr{B}}.$ 

Let $\Top_0 = \{Q_0\}$. Then, assuming that $\Top_k$ has been defined for some $k \geq 0,$ define 
\[ {\Top}_{k+1} = \bigcup_{Q \in \Top_k} \Next(Q) \] 
and let
\[ \Top = \bigcup_{k =1}^\infty {\Top}_k. \]

For each $Q \in \Top$ which is not a singleton, we construct a $C^{1,\alpha}$ surface. 

\begin{remark}\label{r:M}
Let us fix a constant 
\[ M = 2+4C_2 + 4C_1C_2. \] 
By repeated application of the triangle inequality, it is not difficult to show that if $R,Q \in \mathcal{D}$ and $R \sim Q$ then
\[ 2C_1B_Q \subseteq \tfrac{M}{2}B_R \subseteq M^2B_Q. \] 
This fact will be used in the proof of Lemma \ref{l:tree-surface} later. 
\end{remark}

\begin{lemma}\label{l:tree-surface}
Let $Q \in \mathcal{D}$ so that $Q \cap E_0 \not=\emptyset$ and suppose $\tree(Q)$ is not a singleton. For $\delta > 0$ small enough (depending on $\kappa$), there exists a $(1+C\ve)$-bi-Lipschitz map $g_Q$ (invertible on its image) so that $\Sigma_Q = g_Q(\R^d)$ is $C\ve$-Reifenberg flat. If $0 < \alpha < 1,$ then $g_Q$ is $C^{1,\alpha}.$ Furthermore,
\begin{align}\label{eq:closeness}
\sup_{y \in E \cap C_1B_R} \dist(y,\Sigma_R) \lesssim \delta^\frac{1}{d+1}\ell(R) \quad \text{for all} \ R \in \tree(Q).
\end{align}
\end{lemma}

\begin{proof}
The proof consists of checking the hypotheses of Theorem \ref{l:azzam}. Let $C_1,C_2 > 1$ be chosen large enough so that Theorem \ref{l:azzam} holds and let $M$ be as in Remark \ref{r:M}. For each $T \in \tree(Q)$ choose $L_T$ to be a $d$-plane through $x_T$ such that
\[ \beta^{d,1}_E(M^2B_T,L_T) \leq 2 \beta^{d,1}_E(M^2B_T). \] 
In this case, \eqref{e:azzam1} is trivial. Let $T \in \tree(Q).$ We begin by estimating $\ve(T).$ Let $T' \in \tree(Q)$ such that $T' \sim T.$ By Remark \ref{r:M} and our choice of $M$ we have
\begin{align}\label{e:B_T}
C_1B_T \subseteq 2C_1B_T \subseteq \tfrac{M}{2}B_{T'} \subseteq MB_{T'} \subseteq M^2B_T. 
\end{align}
Since $T$ has $(d+1,\kappa)$-separated points, we can apply Lemma \ref{l:angpre} and Lemma \ref{lemma:monotonicity} to get
\begin{align}\label{e:main1}
\begin{split}
d_{C_1B_T}(L_T,L_{T'}) &\lec_{\kappa,C_1,C_2} \beta_E^{d,1}(MB_{T'},L_{T'}) + \beta_E^{d,1}(2C_1B_T,L_T) \\
&\lec \beta_E^{d,1}(MB_{T'},L_{T'}) + \beta_E^{d,1}(MB_T,L_T) \\
&\lec  \beta_E^{d,1}(M^2B_T).
\end{split}
\end{align} 
Taking the max over all $T' \sim T$ gives 
\begin{align}\label{e:epsilon-T}
\ve(T) \lec \beta^{d,1}_E(M^2B_T).
\end{align}
So then, for any $R \in \tree(Q),$ we have 
\[\sum_{R \subseteq T \subseteq Q} \frac{\ve(T)}{\ell(T)^{2\alpha}} \lesssim \sum_{R' \subseteq T \subseteq Q} \frac{\beta^{d,1}_E(M^2B_T)}{\ell(T)^{2\alpha}} < \delta^2.\]
Taking $\delta > 0$ small enough with respect to $\kappa,$ we can apply Theorem \ref{l:azzam} to produce the map $g_Q$ and surface $\Sigma_Q.$ 

Let us prove \eqref{eq:closeness}. By Lemma \ref{lemma:betap_betainfty}, we know 
\[ \beta^d_{E,\infty}(2C_1B_R,L_R) \lesssim \beta^{d,1}_E(4C_1B_R,L_R)^\frac{1}{d+1} \lesssim \beta^{d,1}_E(M^2B_R,L_R)^\frac{1}{d+1} < \delta^\frac{1}{d+1} \]  
for each $R \in \tree(Q).$ This implies
\[ \sup_{y \in E \cap 2C_1B_R} \dist(y,L_R) \lesssim \delta^\frac{1}{d+1}\ell(R) \] 
and \eqref{eq:closeness} holds by Theorem \ref{l:azzam}. 
\begin{comment}
Let $k(T)$ be the largest integer $k$ such that for any $x \in T,$ we have $E \cap M^2B_T \subseteq E \cap B(x,r_k).$ The existence of such a $k$ is guaranteed because $E \subseteq B(0,1)$ and $r_k \rightarrow 0.$ By maximality, it follows that 
\[ \ell(T) \sim r_{k(T)}. \] 
Recall from Remark \ref{r:M} that since $T' \sim T$ we have  $MB_{T'} \subseteq M^2B_T.$ Using Lemma \ref{lemma:monotonicity}, \eqref{e:main1} and the definition of $k(T),$ for any $x \in T$, we have
\begin{align*}
d_{C_1B_T}(L_T,L_{T'}) \lec \beta_E^{d,1}(M^2B_T) \lec \beta_E^{d,1}(x,r_{k(T)}).
\end{align*}

Let us verify \eqref{e:azzam3}. Fix $R \in \tree(Q).$ Since $x_R \in T$ for any $R \subseteq T \subseteq Q$, we can use \eqref{e:epsilon-T} to show
\begin{align*}
\sum_{R \subseteq T \subseteq Q(S)} \frac{\ve(T)^2}{\ell(T)^{2\alpha}} \lesssim \sum_{R \subseteq T \subseteq Q(S)} \frac{\beta^{d,1}_E(x_R,r_{k(T)})^2}{r_{k(T)}^{2\alpha}} \lec \sum_{k =1}^\infty \frac{\beta^{d,1}_E(x_R,r_k)}{r_k^{2\alpha}} < \ve^2,
\end{align*}
where the penultimate inequality follows since $\ell(T)  \sim r_{k(T)}$, so there are only a bounded number of $T$ such that $k(T) = k,$ and the final inequality follows since $x_R \in E_0.$

\end{comment}
\end{proof}

We also have the following packing condition for cubes in $\Top.$ 

\begin{lemma}\label{l:sumB}
\[ \sum_{Q \in \Top } \ell(Q)^d \lesssim \ell(Q_0)^d. \] 
\end{lemma} 

\begin{proof}

Let $Q \in \Top.$ We begin by showing that 
\begin{align}\label{e:Stop-B}
\sum_{R \in \Stop(Q)_\mathscr{B}} \ell(R)^d \lesssim \ell(Q)^d.
\end{align}

Indeed, if $\tree(Q) = \{Q\}$ then $\Stop(Q)_\mathscr{B}$ is just the children of $Q$ in $\mathscr{B}$ (or $Q$ itself, in the case that $Q \in \mathscr{B}$). Since $Q$ has a bounded number of the children, the statement is obvious in this case. Assume then, $\tree(Q)$ is not a singleton and let $\Sigma_Q$ be the surface from Lemma \ref{l:tree-surface}. By \eqref{e:azzam4} we have 
\[\dist(x_R,\Sigma_Q) \lesssim \delta^{\frac{1}{d+1}} \ell(R)\]
for each $R \in \Stop(Q)_\mathscr{B}.$ Hence, for $\delta>0$ small enough $c_0B_R$ carves out a large portion of $\Sigma_Q$ which implies $\mathcal{H}^d(\Sigma_Q \cap c_0B_R) \gtrsim \ell(R)^d.$ Since the $c_0B_R$ are disjoint, this gives \eqref{e:Stop-B} since
\[ \sum_{R \in \Stop(Q)_\mathscr{B}} \ell(R)^d \lesssim \sum_{R \in \Stop(Q)_\mathscr{B}} \mathcal{H}^d(\Sigma_Q \cap c_0B_R) \leq \mathcal{H}^d(\Sigma_Q \cap B_Q) \lesssim \ell(Q)^d. \]

With the help of \eqref{e:Stop-B}, we will show  
\begin{align}\label{e:sum-next}
\sum_{R \in \Next(Q)} \ell(R)^d \leq \frac{1}{2} \ell(Q)^d.
\end{align}
First, by definition,
\begin{align}\label{e:sum-next'}
\sum_{R \in \Next(Q)} \ell(R)^d = \sum_{T \in \Stop(Q)_\mathscr{B}} \sum_{R \in \text{Child}_{k^*}(T)} \ell(R)^d. 
\end{align}
If $T \in \Stop(Q)_\mathscr{B}$ then $T$ does not have $(d+1,\kappa)$-separated points (in the case that $T \in \mathscr{B})$ or there exists a child $T'$ of $T$ so that $T'$ does not have $(d+1,\kappa)$-separated points. This mean that there exists a $(d-1)$-plane $L$ such that 
\[ \beta^{d-1}_{E,\infty}(C_1B_{T'},L) \leq \kappa. \] 
For $C_1$ large enough $T \subseteq C_1B_{T'}$, which implies
\[ \dist(y,L) \leq C_1\kappa \ell(T') = C_1\kappa \rho \ell(T) \] 
for all $y \in T.$ Combing this with our choice for $k^*$ (see \eqref{e:k^*}), we observe that if $R \in \text{Child}_{k^*}(T)$ then 
\[ \dist(x_R,L) \leq  C_1\kappa \rho \ell(T) = C_1\kappa \rho \rho^{-k^*}\ell(R) \leq  \frac{c_0}{2}\ell(R). \]
In particular, $c_0B_R$ carves out a large portion of $L$ which implies 
\[ \mathcal{H}^{d-1}(L \cap c_0B_R) \gtrsim \ell(R)^{d-1}. \]
Since the $c_0B_R$ are disjoint and contained in $B_T$, 
\[ \sum_{R \in \text{Child}_{k^*}(T)} \ell(R)^{d-1} \lesssim \sum_{R \in \text{Child}_{k^*}(T)} \mathcal{H}^{d-1}(L \cap c_0B_R) \leq \mathcal{H}^{d-1}(L \cap B_T) \leq C \ell(T)^{d-1} \]
which after multiplying both sides by $\rho^{k^*}\ell(T)$ gives 
\[ \sum_{R \in \text{Child}_{k^*}(T)} \ell(R)^d \leq C\rho^{k^*}\ell(T)^d. \] 
By taking $\kappa$ small enough we can guarantee that $C\rho^{k^*} < \tfrac{1}{2}$. By plugging this into \eqref{e:sum-next'} and using \eqref{e:Stop-B}, we get
\[ \sum_{R \in \Next(Q)} \ell(R)^d \leq C\rho^{k^*}\ell(Q)^d \leq \frac{1}{2}\ell(Q)^d. \]

Finally, for each $k \geq 1$ we have 
\[ \sum_{Q \in \Top_{k+1}} \ell(Q)^d = \sum_{Q \in \Top_k} \sum_{R \in \Next(Q)} \ell(R)^d \leq \frac{1}{2} \sum_{Q \in \Top} \ell(Q)^d \] 
and so, by induction, 
\[ \sum_{Q \in \Top_k} \ell(Q)^d \leq \left(\frac{1}{2}\right)^k \ell(Q_0)^d. \] 
This proves the lemma since 
\[ \sum_{Q \in \Top} \ell(Q)^d = \ell(Q_0)^d + \sum_{k=1}^\infty \sum_{Q \in \Top_k} \ell(Q)^d \lesssim \ell(Q_0)^d. \]

\end{proof}

\subsection{Tangent points} We partition $E_0$ as follows. For $N \in \N \cup \{\infty\}$, let $E_0^N$ be the set of points in $E_0$ that are contained in exactly $N$ cubes from $\Top$, that is,  

\[ E_0^N = \{x \in E_0 : \sum_{Q \in \Top} \mathds{1}_Q(x) = N \} . \]
Clearly
\begin{align*}
E_0 = E_0^\infty \cup \bigcup_{n \in \N} E_0^N.
\end{align*}

\begin{lemma}\label{l:A^infty}
We have $\mathcal{H}^d(E_0^\infty) = 0.$ 
\end{lemma}

\begin{proof}
Let us define a sequence of covers for $E_0^\infty.$ Let $\Top^\infty$ be those cubes in $\Top$ which contain an infinite number of other cubes from $\Top.$ Let $\Top^\infty_0$ be a maximal collection of cubes in $\Top^\infty.$ By maximality, $\Top^\infty_0$ forms a cover for $E_0^\infty.$ Then, supposing $\Top^\infty_k$ has been defined for some $k \geq 0,$ let $\Top^\infty_{k+1}$ be a maximal collection of cubes from $\Top^\infty$ contained in those cubes from $\Top^\infty_{k}.$ Again, by maximality, $\Top^\infty_k$ forms a cover for $E_0^\infty$ for each $k \geq0.$ Using the fact that $\Top^\infty_i \cap\Top^\infty_j = \emptyset$ for $i \not= j,$ along with Lemma \ref{l:sumB}, we have
\[ \sum_{k = 0}^\infty \sum_{Q \in \Top_k^\infty} \ell(Q)^d \leq \sum_{Q \in \Top} \ell(Q)^d < + \infty \] 
which implies
\[ \lim_{k \rightarrow \infty} \sum_{Q \in \Top^\infty_k} \ell(Q)^d = 0. \] 
Since each $\Top^\infty_k$ forms a cover for $E_0^\infty$ and $\ell(Q) \leq \rho^k \ell(Q_0)$ for all $Q \in \Top^\infty_k,$ we have 
\begin{align*}
\mathcal{H}^d(E_0^\infty) = \lim_{k \rightarrow \infty} \mathcal{H}^d_{\rho^k \ell(Q_0)}(E_0^\infty) \leq \lim_{k \rightarrow \infty} \sum_{Q \in \Top^\infty_k} \ell(Q)^d = 0
\end{align*}
as required.
\end{proof}

\begin{lemma}\label{l:find-stopping}
For each $N \geq 1$ and $x \in E_0^N,$ there exists a cube $Q = Q_x \in \Top$ so that $\tree(Q)$ is not a singleton and $x$ is contained in arbitrarily small cubes from $\tree(Q).$ In, particular, $x \in \Sigma_Q.$ 
\end{lemma}

Before proving Lemma \ref{l:find-stopping}, we need the following. 

\begin{lemma}\label{l:instopB}
Suppose $Q \in \Top$ and let $R \in \Stop(Q).$ If $E_0 \cap R \not=\emptyset$ then $R \in \mathscr{B}.$ 
\end{lemma}

\begin{proof}
We will show
\[ \sum_{R' \subseteq T \subseteq Q} \frac{\beta^{d,1}_E(M^2B_T)}{\ell(T)^{2\alpha}} < \delta^2 \] 
for all siblings $R'$ of $R$. Since $R \in \Stop(Q)$, this implies that $R \in \mathscr{B}.$ 

Let $R'$ be a sibling of $R.$ Since $R \cap E_0 \not=\emptyset$, there exist a point $x \in E_0$ so that if $R' \subset T \subseteq Q$ then $x \in T.$ For each $T,$ Let $k(T)$ be the largest integer $k$ such that $E \cap M^2B_T \subseteq E \cap B(x,r_k).$ The existence of such a $k$ is guaranteed because $E \subseteq B(0,1)$ and $r_k \rightarrow 0.$ By maximality, it follows that 
\begin{align}\label{e:k(T)}
\ell(T) \sim r_{k(T)}.
\end{align}
Also, let $k(R') = k(R^{(1)})$. It is clear then that $E \cap M^2B_{R'} \subseteq E \cap B(x,r_{k(R')})$ and $\ell(R') \sim r_{k(R')}.$ By \eqref{e:k(T)}, for each $k \geq 0$ we have  
\[ \{T \, |\, k(T) = k \} \lesssim 1. \] 
Using this along with Lemma \ref{lemma:monotonicity} and the definition of $E_0$, we get 
\[ \sum_{R' \subseteq T \subseteq Q} \frac{\beta^{d,1}_E(M^2B_T)}{\ell(T)^{2\alpha}} \lesssim \sum_{R \subseteq T \subseteq Q(S)} \frac{\beta^{d,1}_E(x_R,r_{k(T)})^2}{r_{k(T)}^{2\alpha}} \lec \sum_{k =1}^\infty \frac{\beta^{d,1}_E(x_R,r_k)}{r_k^{2\alpha}} < \ve^2.   \] 
Choosing $\ve$ small enough gives the result. 
\end{proof}

\begin{proof}[Proof of Lemma \ref{l:find-stopping}]
Let $N \geq 1$ and let $x \in E_0^N.$ Let $Q \in \Top$ be the smallest cube in $\Top$ so that $Q \ni x.$ By definition of $N$ this means that $Q \in \Top_{N-1}.$ Suppose for the sake of a contradiction that $x$ is not contained in arbitrarily small cubes from $\tree(Q).$ This implies that there exists some $R \in \Stop(Q)$ so that $x \in R.$ By Lemma \ref{l:instopB} this implies $R \in \Stop(Q)_{\mathscr{B}}$ which means there exists a cube $T \in \Top_N$ so that $x \in T.$ Hence $x$ is contained in at least $N + 1$-cubes from $\Top$, which is a contradiction. 

The fact that $x \in \Sigma_Q$ follows \eqref{e:azzam4} and the fact that $x$ is contained in arbitrarily small cubes from $\tree(Q).$ 
\end{proof} 

The following lemma finishes the proof of Proposition \ref{theorem:summabletangent}.

\begin{lemma}\label{l:A_0^N}
Let $N \geq 0.$ Then $\mathcal{H}^d$-almost every $x \in E_0^N$ is a tangent point for $E$.  
\end{lemma}

\begin{proof}

Fix some $N \geq 0.$ Let $\{Q_j\}_{j \in J}$ be the maximal collection of cubes $Q_x$ (from Lemma \ref{l:find-stopping}) such that $x \in E_0^N.$ Notice that if $x \in Q_j$ for some $j \in J$, the $Q_j = Q_x$ otherwise $x$ would be contained in $\geq N+1$ cubes from $\Top.$ By maximality, it is clear that 
\[ E_0^N \subseteq \bigcup_{j \in J} Q_j. \]
Let $j \in J.$ It suffices to show that almost every $x$ in $E_0^N \cap Q_j$ is a tangent point for $E$. We may assume $\mathcal{H}^d(E_0^N \cap Q_j) >0$ since otherwise this is obvious. For brevity, we shall write $Q = Q_j.$ Let $g = g_Q$ and $\Sigma = \Sigma_Q$ be the $C^{1,\alpha}$ map (with $C^{1,\alpha}$ inverse) and surface, respectively, associated to $\tree(Q)$ from Lemma \ref{l:azzam}. By Lemma \ref{l:find-stopping}, we know 
\[ E_0^N \cap  Q \subset \Sigma. \]

At $\hd$-almost every point $x $ of $E_0^N \cap Q$, we can find a tangent $d$-plane for $\Sigma$ (in the sense of \eqref{e:alpha-tangent}). A priori we can't say that such a tangent is a tangent for $E$, however, using the decay of the $\beta$-coefficients at $\mathcal{H}^d$-almost all $x \in E_0$, this turns out to be the case.

Let $K = g^{-1}(E_0^N \cap Q).$ Since $g$ is bi-Lipschitz and $\mathcal{H}^d(E_0^N \cap Q) >0$, we have $0 < \mathcal{H}^d(K) < \infty$. Let $x_0 \in E_0^N \cap Q$ and let $p_0 = g^{-1}(x_0) \in K.$ Assume that $p_0$ is a point of density and that $x_0$ is a tangent point of $\Sigma.$ The set of points satisfying this has full measure in $K$ and its image under $g$ has full measure in $E_0^N \cap Q.$ So, if we can show $x_0$ is a tangent point for $E$, we are done. This is the goal for the remainder of the proof.

Let $\tau >0.$ We know $x_0$ is a tangent point for $\Sigma$ so there exists a $d$-plane $L_0$ and a scale $r_1$ so that if $0 < r < r_1$ then 
\[ \dist(y,L_0) \leq \tau r^{1+\alpha}, \quad \text{for all} \ y \in \Sigma \cap B(x_0,r). \] 
By the Lebesgue density theorem, there exists $r_2 > 0$ so that if $0 < r < r_2$ and $p' \in B_d(p_0,r)$, there exists $p \in K$ such that $|p - p'| \leq \tau r.$ Finally, since $E_0$ is a subset of 
\[ \left\{x \in E \,  | \,  \sum_{k=\ell}^\infty \frac{\beta^{d,1}_E(x,r_k)^2}{r_k^{2\alpha}} < + \infty \right\},\]
there exists $r_3 > 0$ so that if $0 < r< r_3$ then 
\begin{align}
r^{-\alpha}\beta^{d,1}_E(x_0,r) < \tau^{d+1}.
\end{align}
By Lemma \ref{lemma:betap_betainfty}, this implies  
\begin{align}\label{e:betainftyalpha}
r^{-\alpha}\beta^d_{E,\infty}(x_0,r) \lesssim \tau \quad \text{for all} \ 0 < r < r_3/2.  
\end{align}
Let 
\[ r_0 = \min\{ r_1,r_2,r_3/2\}. \] 

For $0 < r < r_0$ let $\{p_i'\}_{i=1}^d$ be a collection of linearly independent points (with good constant) in $B_d(p_0,r).$ Since $r < r_2$ we can find a corresponding collection of points $\{p_i\}_{i=1}^d$ in $K$ so that $|p_i - p_i'| \leq\tau r.$ This implies the $p_i$ are also linearly independent with good constant. For each $i =1,\dots,d$, let $x_i = g(p_i).$ Since $g$ is $(1+C\ve)$-bi-Lipschitz the $\{x_i\}_{i=0}^d$ are linearly independent points (with good constant) in $E_0^N \cap Q \cap B(x_0,2r).$ Since $E_0^N \subseteq \Sigma$ and $r < r_1$, we have 
 \[ \dist(x_i, L_0) \leq \tau r^{1+\alpha} \quad \text{for each} \ i = 0,1,\dots,d. \] 

Let us assume, towards a contradiction that there exists $z \in E \cap B(x_0,r)$ such that 
\[ \dist(z,L_0) > 100\tau r^{1+\alpha}. \] 
This implies there cannot be a $d$-plane $V$ such that $\dist(z,V) \leq C \tau r^{1+\alpha}$ and $\dist(x_i,V) \leq \tau r^{1+\alpha}$ for $i = 0,\dots,d.$ This contradicts the fact that $r^{-\alpha}\beta_{E,\infty}^d(x_0,r) \lec \tau$ -- which holds from the fact that $r < r_3/2$ and \eqref{e:betainftyalpha}
\end{proof}

\begin{corollary} \label{corol:final2}
Let $E \subset \R^{n}$ be a closed set so that $E \subset B(0,1)$. Let $1 \leq p < \infty$. Except for a set of zero $\hd$ measure, if 
\begin{align*}
\int_0^{1} \frac{\betae{p}(x,r)^2}{r^{2\alpha}} \frac{dr}{r} < + \infty,
\end{align*}
then $x$ is a tangent point of $E$.
\end{corollary}
\begin{proof}
It is enough to show that for any $x \in E$,
\begin{align*}
    \sum_{k \in \N} \frac{\betae{p}(x, r_k)^2}{r_k^{2\alpha}} \lesssim \int_0^{1} \frac{\betae{p}(x,r)^2}{r^{2\alpha}}\, \frac{dr}{r}. 
\end{align*}
Let $r_k \leq r \leq 10r_k$. Then recall from Lemma \ref{lemma:monotonicity}, that
\[\betae{p}(x,t)^2 \geq \left( \frac{r_k}{t} \right)^{2(1+d/p)} \betae{p}(x, r_k)^2. \]
Thus,
\begin{align*}
\int_{r_k}^{10r_k} \frac{\betae{p}(x,r)^2}{r^{2\alpha}} \, \frac{dr}{r} \geq \int_{r_k}^{10r_k} \left( \frac{r_k}{r} \right)^{2(\alpha+1+d/p)} \frac{\betae{p}(x, r_k)^2}{r_k^{2\alpha}}\, \frac{dr}{r} 
 \geq \frac{\ln(10)}{10^{2(\alpha+1+d/p)}} \frac{\betae{p}(x, r_k)^2}{r_k^{2\alpha}}. 
\end{align*}
This lets us conclude that
\begin{align*}
    \sum_{k \in \N} \frac{\betae{p}(x, r_k)^2}{r_k^{2\alpha}} \lesssim_{p,d,\alpha} \sum_{k \in \Z} \int_{r_k}^{10r_k} \frac{\betae{p}(x, r)^2}{r^{2\alpha}} \, \frac{dr}{r}
     = \int_0^1 \frac{\betae{p}(x, r)^2}{r^{2\alpha}} \, \frac{dr}{r}.
\end{align*}
The corollary then follows from Proposition \ref{theorem:summabletangent}.
\end{proof}

Given any bounded set $E$, we can translate and dilate it so that it is contained in the unit ball centred at $0$. Hence, Corollary \ref{corol:final2} immediately gives the first implication of Theorem \ref{t:main2}.

\section{Existence of paraboloid points implies fast decay of $\beta$ coefficients}\label{s:tan-beta}
In this section we prove the following two propositions. Proposition  \ref{theorem:tangentbeta} is essentially one direction of Theorem \ref{t:main}, while Proposition \ref{p:tangent-beta-alpha} is one direction of Theorem \ref{t:main2}.
\begin{proposition}\label{theorem:tangentbeta}
Let $\alpha =0$ and $E \subset B(0,1) \subset \R^{n}$. If $d=1$ or $d=2$, let $1 \leq p < \infty$. If $d \geq 3$, let $1 \leq p < \frac{2d}{d-2}$. Then for $\hd$-almost all $\alpha$-tangent points $x \in E$, it holds that 
\begin{align*}
    \sum_{\substack{Q \in \cubes; Q \subset B(0,1) \\ Q \ni x}}\beta_E^{d,p}(Q)^2 < \infty.
\end{align*}
\end{proposition}
\begin{proposition}\label{p:tangent-beta-alpha}
Let $\alpha \in [0, 1)$ let $E \subset B(0,1) \subset \R^{n}$. If $d=1$ or $d=2$, let $1 \leq p < \infty$. If $d \geq 3$, let $1 \leq p < \frac{2d}{d-2}$. Then for $\hd$-almost all $\alpha$-tangent points $x \in E$, it holds that 
\begin{align*}
    \sum_{\substack{Q \in \cubes; Q \subset B(0,1) \\ Q \ni x}}\frac{\overline{\beta}_E^{d,p}(Q)^2}{\ell(Q)^{2\alpha}} < \infty.
\end{align*}
\end{proposition}

\subsection{Preliminaries}
Let $\alpha \in [0, 1)$. Recall the notation from the Introduction: for a  $d$-dimensional plane $V$, a parameter $\theta> 0$ and a point $x \in \R^n$, we set
\begin{align*}
    & X_\alpha(x, V, \theta) := \{ y \in \R^n \, |\, |\Pi_{V^\perp}(x-y) < \sin(\theta) |\Pi_V(x-y)|^{1+\alpha} \} \mbox{ and } \\
    & X_\alpha(x, V, \theta, r):= X_\alpha(x, V, \theta) \cap B(x,r). 
\end{align*}
We call $X_\alpha(x, V, \theta)$ the $\alpha$-paraboloid at $x$ with axis $V$. 
Set also
\begin{align*}
    X_\alpha^c(x, V, \theta, r) := B(x,r) \setminus X_\alpha(x,V,\theta, r). 
\end{align*}
Recall also that point $x \in E$ is called \textit{an $\alpha$-paraboloid point} if the there is a $d$-plane $L$ so that for all $0<s<1$, we can find an $r_0>0$ such that
\begin{align*}
    E \cap B(x,r) \setminus X_\alpha(x, L, \theta) = \emptyset \quad \mbox{ for all } r<r_0. 
\end{align*}

 The proof of the following lemma might be found in \cite{mattila}, Lemma 15.12, for $\alpha=0$ and \cite{del2019geometric}, Lemma 3.3, for $\alpha \in (0, 1)$. 
\begin{lemma} \label{lemma:mattila}
Let $E \subset \R^n$, $V \in G(d, n)$, $\theta \in (0,\pi/2)$ and $0< r< \infty$. If 
\begin{align}\label{eq:cones}
    & E \setminus (X_\alpha(a, V, \theta) \cap B(a, r)) = \emptyset \,\, \mbox{for all} \enskip a \in E, \mbox{ and } \nonumber\\
    & \diam(E) < r,
\end{align}
then $E$ can be covered by one $C^{1, \alpha}$ $d$-dimensional graph.
\end{lemma}
\noindent
%Recall the notation
%\begin{align*}
%& X(a, V, \theta) := \{ x \in \R^n| \dist(x-a, V) < \sin(\theta) |x-a| \} \\
%& X(a, V^\perp, \theta) := \{ x \in \R^n| \dist(x-a, V^\perp) < \cos(\theta) |x-a| \},
%\end{align*}
%where here $a \in \R^n$, $V \in G(d, n)$ and $V^\perp$ is the orthogonal complement of $V$.
%We also put
%\begin{align*}
 %   X(a, V^\perp, \theta, r) := X(a, V^\perp, \theta) \cap B(a, r).
%\end{align*}

%\begin{lemma}
%Let $x \in E$ be a $\alpha$-tangent point. Then $x$ is an $\alpha$-cone point. 
%\end{lemma}
%\begin{proof}
%Let $L_x= L+x$ be the $d$-plane in the definition of $\alpha$-tangent point. 
%If the lemma did not hold, for each $d$-dimensional plane $V \in G(d, n)$ we could find a $\theta \in (0,\pi/2)$ such that for all $r>0$, there existed a $y \in B(x,r) \cap E \setminus X_\alpha(x, V, \theta)$; such a $y$ would have $\dist(x-y, V) \geq \sin(\theta)|x-y|^{1+\alpha}$. One can easily see that this applied to $L$ implies that $\limsup_{r \downarrow 0} \sup_{y \in B(x,r)\cap E} \dist(y,L_x)/r \geq \sin(\theta)$. 
%\end{proof}
%\todo{I included the above lemma in my section} 

\noindent
 Recall Lemma \ref{l:tangent-cone}. Denote by $\dT_\alpha(E)$ the set of $\alpha$-tangent points of $E$ (from now on we will refer to $\alpha$-tangent points simply as `tangent points' whenever possible).
 We can assume without loss of generality that $\hd(\dT_\alpha(E)) >0$, for otherwise there is nothing to prove. If $x \in \dT_\alpha(E)$, denote its tangent by $L_x$. 
Fix $\theta \in (0, \pi/2)$ and let $\{L_k\}_{k\in \N} \subset $ be a dense countable subset of the Grassmannian,  $\{r_\ell\}_{\ell \in \N}$ be a dense, countable subset of $(0,1)$. To ease notation put $E(x, r):= B(x,r) \cap E$ and set
\begin{align}\label{e:Kappa}
    K_{n, \ell}^\alpha:= \left\{x \in \dT_\alpha(E) \, |\, E(x,r_\ell) \setminus X_\alpha\left(x, L_k , \theta, r_\ell\right) = \emptyset \right\}. 
\end{align}
\begin{lemma}
We have
\begin{align*}
    \bigcup_{k,\ell \in \N} K_{n, \ell}^\alpha = \dT_\alpha(E).
\end{align*}
\end{lemma}
\begin{proof}
Indeed, if $x \in \dT_\alpha(E)$, then there exists an $r>0$ and a $d$-plane so that $E(x, r_\ell) \setminus X_\alpha(x, L_x, \theta)$; since $\{L_k\}$ is dense, for each $\delta>0$, we may find a $k\in \N$ so that $\dist_H(L_{k}\cap B , (L_x-x) \cap B) < \delta$, for, say, $B=B(0,1)$. It is clear then, that we may find an $r_\ell \leq r$ so that 
$
    E(x,r_\ell) \setminus X_\alpha(x, L_k, \theta)= \emptyset.
$
\end{proof}
Thus, to prove either Proposition \ref{theorem:tangentbeta} or \ref{p:tangent-beta-alpha} it suffices to prove the following lemmas, respectively. 
\begin{lemma} \label{lemma:betaontan}
Suppose $\alpha=0$. With notation as above,
fix $k, \ell$ and consider $K_{k, \ell}^\alpha$ as above. Let $p$ be given as in Proposition \ref{theorem:tangentbeta}. Then for $\hd$-almost every point $x \in K_{k,\ell}^\alpha$, 
\begin{align*}
    \sum_{\substack{Q \in \cubes, \, \ell(Q) \leq 1 \\ x \in Q}} \betae{p}(Q)^2 < \infty.
\end{align*}
\end{lemma}
\begin{lemma}\label{lemma:betaontan-alpha}
Suppose now $\alpha \in [0, 1)$. With notation as above,
fix $k, \ell$ and consider $K_{k, \ell}^\alpha$ as above. Let $p$ be given as in Proposition \ref{p:tangent-beta-alpha}. Then for $\hd$-almost every point $x \in K_{k,\ell}^\alpha$, 
\begin{align*}
    \sum_{\substack{Q \in \cubes, \, \ell(Q) \leq 1 \\ x \in Q}} \frac{\betae{p}(Q)^2}{\ell(Q)^{2\alpha}} < \infty.
\end{align*}
\end{lemma}

\begin{remark} \label{remark:lipK}
Without loss of generality, we can assume that $K_{k, \ell}^\alpha$ is compact and that $\diam(K_{k, \ell}^\alpha) \leq r_\ell/2$. By Lemma \ref{lemma:mattila}, we have that $K_{k, \ell}^\alpha$ can be covered by the graph of some $C^{1, \alpha}$ function $f : L_k \to L_k^\perp$.
\end{remark}
\noindent
Furthermore, we may (and will) let $L_{k} = \R^d \subset \R^n$, $\Pi_{\R^d} =: \Pi$ and $K_{k, \ell}^\alpha =: K_\alpha$.
Denote by $Q_0$ the minimal cube in $\cubes$ such that $K \subset 3 Q_0$; we may assume that $\diam(Q_0) \leq \frac{1}{3}r_\ell$. By rescaling, we also assume without loss of generality that $r_\ell = 1$. By translating, we take $x_{Q_0} =0$. To ease notation, we set
\begin{align*}
    & X_\alpha(x) := X_\alpha(x, \R^d, \theta, 1); \\
    & X_\alpha^c(x) := B(x,1) \setminus X_\alpha(x, \R^d, \theta, 1).
\end{align*}
Now define
\begin{align}
& \mathcal{S}:= \left\{ Q \in \cubes; Q \subset B(0,1)\, |\, Q \cap K\neq \emptyset \enskip \mbox{ and } Q \subsetneq Q_0 \right\} \label{def:smallfam}\\
& \mathcal{L} := \left\{ Q \in \cubes; Q \subset B(0,1) \, |\, Q \cap K \neq \emptyset \enskip \mbox{ and } \enskip Q \supset Q_0 \right\}.
\end{align}

\subsection{The sum over large cubes}
\begin{lemma} \label{lemma:pointcubes}
We have
\begin{align}\label{eq:cubesest1}
\int_{K} \sum_{\substack{Q \in \cubes, \, \ell(Q) \leq 1 \\ x \in Q}} \frac{\beta_E^{d,p} (Q)^2}{\ell(Q)^{2\alpha}} \, d \hd(x) \lesssim \sum_{Q \in \mathcal{S} \cup \mathcal{L}} \frac{\beta_E^{d,p} (Q)^2}{\ell(Q)^{2\alpha}} \ell(Q)^d.
\end{align}
\end{lemma}
\begin{proof}
We have that
\begin{align} \label{eq:cubesest3}
\int_{K} \sum_{\substack{Q \in \cubes, \, \ell(Q) \leq 1 \\ x \in Q}} \frac{\beta_E^{d,p} (Q)^2}{\ell(Q)^{2\alpha}}\, d \hd(x) & 
= \int_{K} \sum_{Q \in \mathcal{S} \cup \mathcal{L}} \frac{\beta_E^{d,p} (Q)^2}{\ell(Q)^{2\alpha}} \chara_{\{(y, Q) \in K \times \cubes\, | \, y \in Q\}} (x)\, d \hd(x) \nonumber\\
& = \sum_{Q \in \mathcal{S} \cup \mathcal{L}} \frac{\beta_E^{d,p} (Q)^2}{\ell(Q)^{2\alpha}} \int_{K} \chara_{\{y \in K\, |\, y\in Q\}} (x) d \hd(x) \nonumber\\
& \lesssim \sum_{Q \in \mathcal{S} \cup \mathcal{L}} \frac{\beta_E^{d,p} (Q)^2}{\ell(Q)^{2\alpha}}\ell(Q)^d.
\end{align}
The second equality is an application of Fubini-Tonelli's theorem. For the inequality, notice that we are integrating over $K$. Such set is not necessarily Ahlfors regular, for the density may be zero in some balls. However, the upper $d$-regularity is still maintained, since it is a subset of a $d$-dimensional Lipschitz graph. 
\end{proof}
\begin{lemma} \label{lemma:largecubes}
\begin{align} \label{eq:bigcubessum}
    \sum_{Q \in \mathcal{L}} \frac{\beta_E^{d,p} (Q)^2}{\ell(Q)^{2\alpha}} (Q)^2 \ell(Q)^d <\infty.
\end{align}
\end{lemma}
\begin{proof}
Clearly,
$
\betae{p}(Q)^2 \lesssim 1
$ and $\frac{1}{\ell(Q)^{2\alpha}} \lesssim 1$.
Also recall that any $Q \in \mathcal{L}$ is contained in $B(0,1)$. Moreover, for each generation, there can be at most one cube which contains $Q_0$. Hence, the sum \eqref{eq:bigcubessum} is really a finite sum of bounded terms, and therefore must be finite.
\end{proof}
\noindent

The proof now splits between the cases $\alpha=0$ and $\alpha \in (0, 1)$. 
\subsection{The sum over small cubes: Lipschitz case.} Throughout this section, we assume $\alpha=0$ and \textit{we will use the $\beta$ coefficients defined in \eqref{e:beta-matt}}.
We extend $f$ (as given in Remark \ref{remark:lipK}) to the whole $\R^d$; since $\alpha=0$ then $f$ is simply a Lipschitz function, and we use the standard Kirszbraun extension (see \cite{heinonen2005lectures}). 
%If $0<\alpha<1$, we use Theorem 2.7 from \cite{del2019geometric} (which is  \cite{stein1993harmonic} VI.2.3, Theorem 4) to obtain a $C^{1, \alpha}$ extension. In either cases,
We denote the graph map of the extension by $g: \R^n \to \R^n$, and put
\begin{align}
    \Gamma := g(\R^d).\nonumber
\end{align}
(So $\Gamma$ is the graph). 
Before going any further, we construct a Whitney decomposition of the `bad set' $K_0^c$ (recall the notation in \eqref{e:Kappa}), which will then be used to estimate the sum over small cubes. We put
\begin{align*}
    G_0:=\Pi(K_0) \subset \R^d.
\end{align*}
Since $G_0$ is compact, the complement $B_d(0,1) \setminus G_0$ (here $B_d$ is a ball in $\R^d$) is open and thus admits a Whitney decomposition (see \cite{grafakos2008classical}, Appendix J) which we denote by $\Ww_d$.
%The choice of $\alpha$ in the construction of the Whitney decomposition is obviously the same as the $\alpha$ given in the statement of Poposition \ref{theorem:tangentbeta}.
%To ease notation, put
%\begin{align*}
 %   X_\alpha^c(x, V, \theta, r) := B(x,r) \setminus X_\alpha(x, V, \theta, r).
%\end{align*}
%Now set
%\begin{align}
%    X_\alpha=X_\alpha(K):= \bigcup_{x \in K} \Int(X^c_\alpha(x, \R^d, \theta, 1)),
%\end{align}
%and
%\begin{align*}
%    F_\alpha:= B(0,1) \setminus X_\alpha.
%\end{align*}
%By definition, $X^_\alpha(x, V, \theta)$ is open. Its intersection with $B(x,r)$ is open, too. Then $X^c_\alpha(x,V, \theta, r)$ is closed (relative to $B(x,r)$) and $\Int(X^c_\alpha(x,V, \theta, r))$ is open. Thus $X_\alpha$ is open and so 
Now define
\begin{align*}
    & E_0:= \bigcap_{x \in K_0} \overline{X_0(x)}, \mbox{ and} \\
    & F_0 := E_0\setminus K_0
\end{align*}
where the closure is taken with respect to the topology of $B(0,1)$. Note that $E_0$ is closed in $B(0,1)$. Also $K_0 \subset E_0$. 
For each $S \in \Ww_d$, put
\begin{align*}
    T_S := \Pi^{-1}(S) \cap E_0 \subset \R^n.
\end{align*}
The above mentioned `Whitney decomposition' is given by\footnote{We constructed the decomposition in this way - rather than working directly in $\R^n$, because we want the diameters of $T_S$ to be proportional to their distance to $K_0$.}
\begin{align*}
    \Ww_n:= \left\{ T_S \, |\, S \in \Ww_d \right\}.  
\end{align*}
It is immediate that for each $T_S \in \Ww_n$, $T_S \cap K_0 = \emptyset$, for otherwise there would be a point $x \in K_0$ which project into $S$. It is also true that   
\begin{align} \label{eq:Tswhitney0}
    E_0= K_0 \cup F_0=  K_0 \cup \left( \bigcup_{S \in \Ww_d} T_S\right),
\end{align}
where the union is disjoint.
Indeed, if $y \in F_0 = E_0 \setminus K_0$ then also $\Pi(y) \neq \Pi(x)$ for all $y \neq x \in K$ (since otherwise $y$ would not be in $\overline{X_0(x)}$ and thus not belong to the intersection $E_0$). So $\Pi(y) \in \R^d \setminus G$ and thus $\Pi(y) \in S$ for some $S \in \Ww_d$. Then $y \in \Pi^{-1}(S) \cap E_0$. The reverse inequality is immediate.
Note also that, if $P, S \in \Ww_d$ are so that $P \cap S = \emptyset$, then 
\begin{align} \label{eq:Tswhitney1}
    T_S \cap T_P = \emptyset.
\end{align}
\noindent
Note that $E\cap B(0,1) \subset E_0$.
Before estimating the sum of $\beta$ coefficients over small cubes, we prove some preliminary lemmas. 

\begin{lemma}\label{l:Gamma-in-Eplus}
With the notation above, we have
\begin{align*}
    \Gamma \cap B(0,1) \subset E_0. 
\end{align*}
\end{lemma}
\begin{proof}
For the sake of contradiction, let $x$ be a point with $x \in \Gamma \cap B(0,1)$ with $x \notin E_0$. Then $x \notin K_0$ and $x \notin T_S$ for any $T_S \in \Ww_n$. But this implies that $x \in \Int(X_0^c(y))$ for some $y \in K_0$. This violates the fact that the Kirszbraun extension of $f$ has the same Lipschitz constant as $f$, and in particular the same cones. 

%If $0 \in (0, 1)$, the extension is again a $C^{1, \alpha}$ with a possible change in norm by a universal constant (independent of $\Pi(K_\alpha)$). If needed, we may account for this change by modifying\footnote{In the following we don't make this modification explicit.} the parameter $\theta$ of the $\alpha$-paraboloids in the definition of $E_0$. Then the proof given for the case $\alpha=0$ works in this case, too. 
\end{proof}

\begin{lemma}\label{l:punto-bordo}
Let $p \in F_0$. Then there exists a point $x \in K_0$ so that $(x,p] \subset F_0$, where $[x, p]$ is the straight line segment joining $x$ to $p$ and $(x, p] = [x, p] \setminus \{x\}$. 
\end{lemma}
\begin{proof}
Suppose that the claim is false, and let $C$ be the connected component of $F_0$ containing $p$. If $C=\{p\}$, then since $K_0$ is compact, we can find an open neighbourhood $N(p) \subset B(0,1)$ of $p$ so that $N(p) \cap K_0= \emptyset$. Moreover, we can find a point $y \in (x, p] \cap B(0,1) \setminus E_0$ arbitrarily close to $p$. In particular, there exists an $x' \in K_0$ so that $y \in X_0^c(x')$. Now we have two possibilities: either $x' \notin X(x)$ or $x' \in X(x)$. The first case cannot hold, since $x',x \in K_0$. The second case cannot hold either, since if it did, and since we may choose $y$ as close to $p$ as we wish, it would imply that $p \in X_0^c(x')$, which is impossible since $p \in E_0$. This implies that the lemma must hold in the case $C=\{p\}$. If $C$ is not a singleton, the argument is the same, and we leave it to the reader.

%Then we can find a $p \in F_\alpha$ for which no point $x \in K_\alpha$ can be found with $(x, p] \subset C$. Let $x_1 \in K_\alpha$ be a closest point to $p$. By assumption, there must be another point $ x_2 \in K_\alpha \setminus C$ so that
%\begin{align*}
 %    X_\alpha^c(x_2, \R^d, \theta, 1) \cap [x_1, p] \neq \emptyset.
%\end{align*}
%But this implies that $x_2 \in X_\alpha^c(x_1, \R^d, \theta, 1)$, which contradicts the fact that $x_1, x_2 \in K_\alpha$. 
\end{proof}

\begin{lemma}\label{l:Eplus}
The set $E_0$ is lower content $d$-regular. 
\end{lemma}
\begin{proof}
Let $0<r<1$. If $x \in K_0$, then $\hdc(E_0 \cap B(x,r)) \geq \hdc(\Gamma \cap B(x,r)) \gtrsim r^d$, since $\Gamma$ is a bi-Lipschitz image of a plane. 
If on the other hand $x \in F_0$, let $x_1$ be a closest point in $K_0$ so that $(x_1,x] \subset F_0$. Such a point exist by Lemma \ref{l:punto-bordo}. Fix a constant $\cc_1 < 1/4$. We consider two cases.

\textit{Case 1.} We have $|x-x_1| < \cc_1 r$. Then $B(x_1, r/2) \subset B(x,r)$, and thus, using also Lemma \ref{l:Gamma-in-Eplus},
\begin{align*}
    \hdc(E_0 \cap B(x,r)) \geq \hdc(E_0 \cap B(x_1, r/2)) \geq \hdc(\Gamma \cap B(x_1, r/2)) \gtrsim r^d. 
\end{align*}

\textit{Case 2.} Let now $|x-x_1| \geq \cc_1 r$. Consider the annulus
\begin{align*}
    A(x):= B(x, \cc_1/2 r) \setminus B(x, \cc_1/4 r). 
\end{align*}
Note that at least half of the volume of $A(x)$ is contained in $X_0(x_1)$\footnote{in the specific case when $x$ is on the boundary of $X_0(x)$, $d=1$ and $n=2$ for example.}. Then we split up $A(x)$ into sectors of diameter comparable to $r$ (with constant depending on $\cc_1$). If any one of these sectors are fully contained in $F_0$, we are done. So suppose that each of them contains at least one point $p \in E_0^c$. Fix the sector $S$ for which $\dist(x_1, S)$ is minimal among all sectors. This sector will have the property that
\begin{align}\label{e:Ealphalcr-1}
    \sup_{q \in S} |q-x_1| < \frac{\cc_1}{2} |x-x_1|. 
\end{align}
Let $x_3 \in K_0$ be the point for which $p \in X_0^c(x_3) \cap S$ (such a point exists for otherwise $p \in E_0$). But then it follows from \eqref{e:Ealphalcr-1} that also $x \in X_0^c(x_3)$. This contradicts the fact that $x \in E_0$ and we are done.  
%Let us pick a ball $B$ so that $B\subset X(x_1)$, $x \in \partial B$ and $r(B) \geq \cc_2 r$, for some $\cc_2$ sufficiently small and to be fixed later. Suppose now that for any choice of one such ball we find a point $p \in E_\alpha^c$. We now specifically choose a ball be satifying the above properties so that $x

%, then we make the following observation.  Note that the complementary cone $X_\alpha^c(x):= X_\alpha^c(x, \R^d, \theta,1)$ cannot contain $x' \in K_\alpha$. Moreover, we can find a ball $B$ with $r(B) \approx \cc_1 r$ so that $B \subset (X_\alpha^c(x) \cap B(x, \cc_1r) \setminus X$. Clearly then
%\begin{align*}
  %  \hdc(E_0 \cap B(x,r)) \geq \hdc(B) \gtrsim_{\cc_1} r^d.
%\end{align*}
%This ends the proof of the lemma. 
\end{proof}

%If $\hdc(K\cap B) \geq c r^d$, then there is nothing to prove. So suppose that $\hdc(K \cap B) < c r^d$; then also
%\begin{align}\label{e:smallproj1}
%    \hdc(\Pi(K \cap B))< c r^d.
%\end{align}
%Moreover, $\Pi(B)= B_d$, the ball in $\R^d$ of radius $r$ and with center $x$. We have
%\begin{align*}
%    r^d \approx \hdc(B_d) \lesssim \hdc(\Pi(K \cap B)) + \hdc(B_d \setminus \Pi(K \cap B)),
%\end{align*}
%which, together with \eqref{e:smallproj1} and with an appropriate choice of $c>0$, implies $r^d \lesssim_c \hdc(B_d \setminus \Pi(K \cap B))$. Let $p \in B_d \setminus \Pi(K \cap B)$. We claim that $\Pi^{-1}(p)$ must contain at least one point $y$ belonging to $F$. Suppose not. Then since $\Pi^{-1}(p) \cap K= \emptyset$ by definition, then $\Pi^{-1}(p)$ must be entirely contained in a double cone of the form $X(y, (\R^d)^\perp, \theta, 1)$, $y \in K$, which was used to construct $X$. But then $\Pi^{-1}(p)$ must also meet the vertex of $X(y, (\R^d)^\perp, \theta, 1)$ - and this contradicts our hypothesis that $p \notin \Pi(K \cap B)$. 
%\end{proof}

The next lemmas show that $\Ww_n$ actually behaves like a Whitney decomposition. 
\begin{lemma}\label{l:diamTS}
For $T_S \in \Ww_n$, 
\begin{align}
\diam(T_S) \lesssim_{\Lip(f)} \diam(S).
\end{align}
\end{lemma}
\begin{proof}
Let $y,z \in T_S$. 
Pick $p=p(y) \in G_0 = \Pi(K_0)$ so that 
\begin{align} \label{eq:choice2}
    |p - \Pi(y)| \leq 2 \dist(\Pi(y), G_0),
\end{align}
and set 
\begin{align} \label{eq:choice1}
    x(y) := g(p) \in K_0.
\end{align}
We choose $q=q(z) \in G_0$ and $x(z)= g(q) \in K_0$ in the same way. 
Further, let $L_y$ (resp. $L_z$) be the $d$-plane parallel to $\R^d$ which contains $x(y)$ (resp. $x(z)$). Denote by $\Pi_y$ (resp. $\Pi_z$) the orthogonal projection onto $L_y$ (resp. $L_z$). 
Then set
\begin{align}
   & \wt y := \Pi_y (y) \quad \mbox{ and } \wt z := \Pi_z (z). \label{eq:choice4} 
\end{align}
Then,
$
    |y - z| \leq |y - \wt y| + |\wt y - \wt z| + |\wt z- z|. 
$
We see that
\begin{align*}
    | y - \wt y|  = |(y-x(y)) - (\wt y - x(y))| 
     = \dist(y - x(y), \R^d).
\end{align*}
Since $x(y) \in K_0$ and $y \in B(0,1)$ (and thus in particular $|x(y)-y| < r_\ell = 1$), then 
\begin{align} \label{eq:dist1}
    \dist(y- x(y), \R^d) \leq \sin(\theta) |y-x(y)|,
\end{align}
and also
\begin{align} \label{eq:dist2}
    \dist(y - x(y), (\R^d)^\perp)  \geq \cos(\theta) |y - x(y)|. 
\end{align}
Furthermore, notice that 
\begin{align} \label{eq:dist3}
    \dist( y-  x(y), (\R^d)^\perp) = |\Pi(x(y)) - \Pi(y)|.
\end{align}
Thus, \eqref{eq:dist1}, \eqref{eq:dist2}, \eqref{eq:dist3} and the choice of $x(y)$ (as in \eqref{eq:choice2} and \eqref{eq:choice1}), give
\begin{align*}
    |y- \wt y| & \leq \sin(\theta) |x(y)-y| 
     \leq \frac{\sin(\theta)}{\cos(\theta)} \dist(y - x(y), (\R^d)^\perp) \\
    & = \frac{\sin(\theta)}{\cos(\theta)} |\Pi(x(y))- \Pi(y)| 
    \leq \frac{1}{\cos(\theta)} |p - \Pi(y) | \\
    & \leq 2 \Lip(F)  \dist(\Pi(K), \Pi(y)).
\end{align*}
Since $\Pi(y) \in S$, and $S$ is a Whitney cube, we see that 
\begin{align*}
    \dist(\Gamma, \Pi(y)) \leq \dist(\Gamma, S) + \diam (S)  \lesssim \diam(S).
\end{align*}
We can conclude that
\begin{align} \label{eq:dist4}
    |y - \wt y| \lesssim_{\Lip(f)} \diam(S).
\end{align}
The same argument gives $|z- \wt z| \lesssim_{\Lip(f)} \diam(S)$. 
We need to estimate $|\wt y - \wt z|$. We have
\begin{align*}
    |\wt y- \wt z| \leq |\wt y - g(\Pi(\wt y))|+ |g(\Pi(\wt y)) - g(\Pi(\wt z))| + |g(\Pi(\wt z))- \wt z|. 
\end{align*}
Now, 
\begin{align} \label{eq:lemmawhit2}
    |g(\Pi(\wt y)) - g(\Pi(\wt z))| \leq \lip(f) |\Pi(\wt y)- \Pi(\wt z)| \leq \Lip(f) \diam(S),
\end{align}
since $\Pi(\wt y)$, $\Pi(\wt z) \in S$.
On the other hand, again using the choice of $x(y)$,  
\begin{align} \label{eq:lemmawhit3}
    |\wt y - g(\Pi(\wt y))|  &\leq |\wt y - x(y)| + |x(y) - g(\Pi(\wt y))| \nonumber\\
    & = |\Pi(x(y))- \Pi(y)| + |g(p) - g(\Pi(y))| \nonumber\\
    & = |p - \Pi(y)| + |g(p) - g(\Pi(y))| \nonumber\\
    & \lesssim \Lip(f) \diam(S). 
\end{align}
This, together with \eqref{eq:lemmawhit2} and \eqref{eq:dist4}, give the lemma. 
\end{proof}

\begin{lemma}\label{l:dist-1}
For $T_S \in \Ww_n$,
\begin{align} \label{eq:whitney_diams2}
\dist(T_S, K_0) \approx \diam(S).
\end{align}
\end{lemma}
\begin{proof}
Since $\Pi$ is $1$-Lipschitz and $S$ is a Whitney cube, 
\begin{align*}
\dist(T_S, K) \geq \dist(S, \Pi(K)) \approx \diam(S).
\end{align*}
\noindent
On the other hand, if we let $y \in T_S$, $x(y)$ as in \eqref{eq:choice2} and \eqref{eq:choice1}, we see that
\begin{align}\label{eq:dist5}
    \dist(T_S, K) & \leq |y - x(y)| 
     \leq \frac{1}{\cos(\theta)} \dist(y-x(y), (\R^d)^\perp) \nonumber\\
    & = \frac{1}{\cos(\theta)} |\Pi(x(y))-\Pi(y)| 
 = \frac{1}{\cos(\theta)} |p- \Pi(y)|\nonumber \\
    & \lesssim_{\Lip(f)} \diam(S).
\end{align}
\end{proof}

\noindent
We conclude, also using \eqref{eq:Tswhitney0} and \eqref{eq:Tswhitney1} that $\Ww_n$ is a disjoint decomposition of $Q_0 \setminus K$ so that
\begin{align} \label{eq:whitney_diams}
     \hdc(T_S)\lesssim_{\Lip(f)} \hdc(S).
\end{align}
\noindent
We state one last lemma before estimating the $\beta$ coefficients. 
\begin{lemma} \label{lemma:distF}
Let $y \in T_S \in \Ww_n$. Then
\begin{align*}
    \dist(y, \Gamma) \lesssim_{\Lip(f)} \diam(S).
\end{align*}
\end{lemma}
\begin{proof}
For $y \in T_S$, let $x(y)$ and $p=p(y)$ as above. Then
$
    \dist(y, \Gamma)   \leq |y - g(p)| 
     = |y - x(y)|.
$
If we now argue as in \eqref{eq:dist5}, the lemma follows.
\end{proof}

\noindent
We are now set to prove Proposition \ref{lemma:betaontan} when $E$ is not necessarily $d$-LCR, and $\alpha=0$. First,
we apply Lemma \ref{l:error2} to $E$ and  $E_0$, and second we apply Lemma \ref{lemma:azzamschul} to $E_0$ and $\Gamma$; recall that a hypothesis of Lemma \ref{l:error2} is that the `target' set is lower content regular. In Lemma \ref{lemma:azzamschul} we need both sets $E_1$ and $E_2$ (with notation as in the statement) to be lower content $d$-regular: we introduced $E_0$ to serve as a `bridge' between $E$ and $\Gamma$. 

Let $B$ be a ball centered on $E_0$. A second hypothesis in Lemma \ref{lemma:azzamschul} is that we can `transfer' the $\beta$'s to a ball $B'$ centered on $\Gamma$ if $B'$ is so that $r(B)=r(B')$ and $B \subset 2B'$. Let us see why this second hypothesis is satisfied. For any $Q \in \dS$, we consider $B=B_Q$; being centered on $E$, $B$ is also centered on $E_0$. Then, since $Q \in \dS$, we find a point $y \in B_Q \cap K \subset B_Q \cap \Gamma$. Then the ball $B':=B(y, r(B))$ is centered on $\Gamma$, and obviously has $r(B)=r(B')$ and $2B' \supset B$. Given a ball $B \in \dS$, let $L_{2B'}$ be the best approximating plane for $\Gamma$ in $2B'$. With this notation and using Lemma \ref{l:error2}, we first estimate\footnote{Here $c$ is the lower regularity constant of $E_0$ as in Lemma \ref{l:Eplus}.}
\begin{align*}
     \beta_{E}^{d, p}(Q)^2 & \leq  \beta_E^{d, p}(B, L_{2B'})^2
     \lesssim_{c, d, p} \beta_{E_0}^{d, p}(2B, L_{2B'})^2 +  \left( \frac{1}{r(B)^d} \int_{2B \cap E_0} \left( \frac{\dist(x, E)}{r(B)}\right)^p \, \hdc(x)\right)^{\frac{2}{p}} \\
    & =: C(d, p, c) \left[\beta_{E_0}^{d,p}(2B, L_{2B'})^2 + E_1^{d,p}(2B, E_0, E)^2\right].
\end{align*}
Note that, since $E_0$ is lower content regular with dimension $d$ (Lemma \ref{l:Eplus}), then 
\begin{align*}
    \beta_{E_0}^{d, p}(2B, L_{2B'}) \approx \overline{\beta}_{E_0}^{d, p}(2B, L_{2B'}).
\end{align*}
Second, using Lemma \ref{lemma:azzamschul}, we see that
\begin{align} \label{eq:Error} \overline{\beta}_{E_0}^{d, p}(2B, L_{2B'})^2 
     \lesssim \overline{\beta}_{\Gamma}^{d, p}(2B')^2 + & \ps{\frac{1}{\ell(Q)^d}\int_{E_0 \cap 4B} \ps{\frac{\dist(y, \Gamma)}{\ell(Q)}}^p \, d \hdc(y) }^{\frac{2}{p}}\nonumber \\ & \quad \quad \quad \quad \quad \quad
    =: C(d,p, c) \left[ \overline{\beta}_\Gamma^{d, p} (2B')^2 + E_2^{d, p}(4B, \Gamma, E_0)^2\right]. 
\end{align}
Now, it is easily seen that for each $B'$ given as above, there exists a bounded number of Christ-David cubes for $\Gamma$ (constructed as in Theorem \ref{theorem:christ}) which intersect it and with approximately the same radius. Using this, and Lemma \ref{lemma:monotonicity}, we have 
\begin{align*}
    \sum_{Q \in \dS} \overline{\beta}_{\Gamma}^{d, p}(2B')^2 \ell(Q)^d \approx \sum_{\substack{ Q \in \mathcal{D}_\Gamma\\ \ell(Q) \leq 1}} \overline{\beta}_{\Gamma}^{d, p}(Q) \ell(Q)^d.
\end{align*}
Since $\Gamma$ is a Lipschtiz graph (thus a bi-Lipschitz image of a $d$-plane), it follows from Dorronsoro's theorem (\cite{dorronsoro1985characterization}) that the David-Semmes (see \cite{david-semmes91}) $\beta$-numbers (which we denote by $\widetilde{\beta}_{\Gamma}^{d, p}$) define a `Carleson measure'. This in particular implies that
\begin{align}\label{eq:smallLip}
    \sum_{\substack{ Q \in \mathcal{D}_\Gamma \\ Q \subset B(0,1)}} \overline{\beta}_\Gamma^{d,p}(Q)^2 \ell(Q)^d \approx \sum_{\substack{ Q \in \mathcal{D}_\Gamma \\ Q \subset B(0,1)}} \widetilde{\beta}_\Gamma^{d,p}(Q)^2 \ell(Q)^d \lesssim 1. 
\end{align}
\begin{lemma}\label{l:error-1}
We have
\begin{align}
    \sum_{Q \in \dS} E_1^{d,p}(2B, E_0, E)^2 \ell(Q)^d < + \infty.
\end{align}
\end{lemma}
\begin{proof}
By Jensen's inequality (Lemma \ref{l:jensen}), we assume that $p \geq 2$, since $E_1^{d,p}(2B, E_0, E) \lesssim E_1^{d,2}(2B, E_0, E)$ for any $1 \leq p \leq 2$. Then, for $Q \in \dS$, and keeping the notation as above, we compute 
\begin{align*}
    \int_{2B \cap E_0} \left( \frac{\dist(y, E)}{\ell(Q)}\right)^p \, \hdc(y) \leq \int_{2B \cap (E_0 \setminus E)} \left( \frac{\dist(y, E)}{\ell(Q)}\right)^p \, \hdc(y).
\end{align*}
If $y \in (E_0 \setminus E)\cap 2B$, then $y \in F$, and thus $y \in T_S$ for some $S \in \Ww_n$. Thus, by Lemma \ref{l:dist-1}, we have 
\begin{align*}
    \dist(y, E) \leq \dist(y, K) \lesssim \diam(T_S).
\end{align*}
We then compute
\begin{align*}
    \int_{2B \cap (E_0 \setminus E)} \left( \frac{\dist(y, E)}{\ell(Q)} \right)^p \, d \hdc(y) & \leq \sum_{\substack{T_S \in \Ww_n \\ T_S \cap 2B \neq \emptyset}} \int_{T_S} \left( \frac{\dist(y, E)}{\ell(Q)}\right)^p \, d \hdc(y)  \\
    & \lesssim \sum_{\substack{T_S \in \Ww_n \\ T_S \cap 2B \neq \emptyset}} \frac{\hdc(T_S) \diam(S)^p}{\ell(Q)^p}.
\end{align*}
Then we get that 
\begin{align}
    & \sum_{Q \in \mathcal{S}} 
    \ps{ 
    \sum_{\substack{T_S \in \Ww_{n}\\ T_S \cap 2B_Q \neq \emptyset}} \frac{\hdc(T_S) \diam(S)^p}{\ell(Q)^p} \ell(Q)^{d\ps{\frac{p-2}{2}}}
    }^{\frac{2}{p}} \nonumber\\
     & \lesssim \sum_{Q \in \mathcal{S}} 
     \sum_{\substack{T_S \in \Ww_n\\ T_S \cap 2B_Q \neq \emptyset}} \left( \frac{\diam(S)^{p+d}}{\ell(Q)^{p+d}}\right)^{\frac{2}{p}} \, \ell(Q)^d 
    \lesssim \sum_{T_S \in \Ww_n}  \sum_{\substack{Q \in \mathcal{S} \\  2B_Q \cap T_S \neq \emptyset}} \frac{\diam(S)^{\frac{2d}{p} + 2}}{\ell(Q)^{\frac{2d}{p} - d +2}} \label{eq:whitney2}
\end{align}
Notice that, 
\begin{align} \label{e:summing1}
    \mbox{if } d=1 \mbox{ or } d=2\, \mbox{ then }\enskip  2-\frac{2}{p}d\ps{\frac{p-2}{2}} > 0 \enskip \forall p \geq 2.
\end{align}
On the other hand,
\begin{align}\label{e:summing2}
    \mbox{if } d \geq 3 \mbox{ then } \enskip 2-\frac{2}{p}d\ps{\frac{p-2}{2}} > 0  \mbox{ for } 2 \leq p < \frac{2d}{d-2}.
\end{align}
Given the hypotheses in Proposition \ref{theorem:tangentbeta}, in either case we may set
\begin{align}\label{e:summing3}
    2-\frac{2}{p}d\ps{\frac{p-2}{2}} =: \alpha > 0.
\end{align}

Now let $z \in T_S \cap B_Q$. Then we see, using \eqref{eq:whitney_diams2} and recalling that, by definition of $\mathcal{S}$, $K \cap Q\neq \emptyset$,
\begin{align} \label{e:new1}
\diam(S) \approx \dist(T_S, K)  \lesssim |z - c_Q| 
\lesssim \ell(Q),
\end{align}
where $c_Q$ is the center of $B_Q$.
Thus, given Whitney cube $T_S \in \Ww_n$, the number of cubes $Q \in \dS$ so that  $T_S \cap 2B_Q \neq \emptyset$ is bounded above by a universal constant $C$. Thus we have
\begin{align}
   \sum_{\substack{Q \in \mathcal{S} \\ 2B_Q \cap T_S \neq \emptyset}} \frac{\diam(S)^{\frac{2d}{p} + 2}}{\ell(Q)^{\frac{2d}{p} - d +2}} & = \diam(S)^{\frac{2d}{p} + 2} \sum_{\substack{Q \in \mathcal{S} \\ 2B_Q \cap T_S \neq \emptyset}} \frac{1}{\ell(Q)^{\frac{2d}{p} - d +2}}\nonumber \\
& \lesssim_C   \frac{\diam(S)^{\frac{2d}{p} + 2}}{\diam(S)^{\frac{2d}{p} - d +2}}
   \approx \diam(S)^d.\label{e:summing4}
\end{align}
Hence, we obtain
\begin{align}\label{e:summing5}
    \eqref{eq:whitney2} & \lesssim \sum_{T_S \in \Ww_n} \hdc(T_S)
    \lesssim \sum_{S \in \Ww_d} \hdc(S) 
     \lesssim 1.
\end{align}
This completes the proof of the lemma.
\end{proof}

\begin{lemma} \label{lemma:errorBbound}
We have
\begin{align*}
    & \sum_{Q \in \dS} E_2^{d,p}(4B, E_0, \Gamma)^2 \ell(Q)^d \\
    & \quad \quad \quad = 
     \sum_{Q \in \dS} \ps{\frac{1}{\ell(Q)^d}\int_{E_0 \cap 4B_Q} \ps{\frac{\dist(y, \Gamma)}{\ell(Q)}}^p \, d \hdc(y) }^{\frac{2}{p}} \ell(Q)^d < + \infty.
\end{align*}
\end{lemma}
\begin{proof}
As before, we assume $p \geq 2$. Note that if $y \in E_0 \setminus \Gamma$, then $y \in F$. Thus, if $y \in E_0 \setminus \Gamma$, then $y \in T_S \in \Ww_n$. Using the fact that $\dist(y, \Gamma) \lesssim \diam(T_S) \approx \diam(S)$ (Lemma \ref{l:dist-1}) we compute 
\begin{align*}
\int_{\ps{E_0 \setminus \Gamma}\cap 2B} \ps{\frac{\dist(y, \Gamma)}{\ell(Q)}}^p \,  d\hdc(y) & \leq \sum_{\substack{T_S \in \Ww_{n} \\ T_S \cap 2B \neq \emptyset}} \int_{T_S)} \ps{\frac{\dist(y, \Gamma)}{\ell(Q)}}^p \, d \hdc (y)\\
& \lesssim \sum_{\substack{T_S \in \Ww_{n}\\ T_S \cap 2B \neq \emptyset}} \frac{\hdc(T_S) \diam(S)^p}{\ell(Q)^p}
\end{align*}
The remainder of the proof is precisely as in Lemma \ref{l:error-1}. 
\end{proof}

Now \eqref{eq:cubesest3} together with Lemma \ref{lemma:largecubes}, Lemma \ref{l:error-1}, Lemma \ref{lemma:errorBbound} and \eqref{eq:smallLip} give
\begin{align*}
\int_{K} \sum_{\substack{Q \in \cubes, \, \ell(Q) \leq 1 \\ x \in Q}} \betae{2}(Q)^2 \, d \hdc(y) < \infty
\end{align*}
and therefore
\begin{align*}
    \sum_{\substack{Q \in \cubes, \, \ell(Q) \leq 1 \\ x \in Q}} \beta_E^{d,2}(Q)^2 < \infty \enskip \enskip \mbox{for} \enskip \enskip \hd \mbox{-almost all} \enskip x \in K. 
\end{align*}
This concludes the proof of Lemma \ref{lemma:betaontan} and therefore of Proposition \ref{theorem:tangentbeta}.

\begin{corollary} \label{corol:final}
Let $E \subset B(0,1) \subset \R^n$. If $d=1$ or $d=2$, let $1 \leq p < \infty$. If $d \geq 3$, let $1 \leq p < \frac{2d}{d-2}$. Then, except for a set of zero $\hd$ measure, if $x \in E$ is a tangent point of $E$ then
\begin{align*}
\int_0^1 \beta_E^{d, p} (x,t)^2 \, \frac{dt}{t} < \infty.
\end{align*}
\end{corollary}
\begin{proof}
Let $x \in E$ be a tangent point of $E$.
Recall Lemma \ref{lemma:monotonicity}. Then, for $1\leq p < \infty$,
\begin{align}
    \int_0^1 \beta_E^{d,p} (x,t)^2 \, \frac{dt}{t} & \lesssim \int_0^1 \beta_E^{d,p}(Q)^2 \, \frac{dt}{t} \nonumber\\
    & \sim \sum_{Q \ni x} \beta_E^{d,p}(Q)^2 \, \int_0^1 \chara_{\{t \in [0,\infty)| C^{-1}\ell(Q) \leq t \leq C \ell(Q)\}} \, \frac{dt}{t} \nonumber\\
    & = \sum_{Q \ni x} \beta_E^{d,p}(Q)^2 \, \int_{C^{-1}\ell(Q)}^{C \ell(Q)} \, dt/t \nonumber\\
    & = \ln (C^2) \sum_{Q \ni x} \beta_E^{d,p}(Q)^2.  \label{eq:corollary}
\end{align}
This sum is bounded by Theorem \eqref{theorem:tangentbeta}.
\end{proof}

Now, any bounded set $E \subset \R^n$ can be dilated and translated so that it is contained in $B(0,1)$. Thus, Theorem \ref{t:main} follows immediately from Corollaries \ref{corol:final2} and \ref{corol:final}.

\subsection{The sum over small cubes: $C^{1, \alpha}$ case.}
Throughout this subsection, we let $\alpha \in [0,1)$ and \textit{we will use the $\overline{\beta}$ coefficients defined in \eqref{e:beta-content}}. 

Recall the notation $K_\alpha$ in \eqref{e:Kappa} and that we want to prove Lemma \ref{lemma:betaontan}. We use Lemma \ref{lemma:mattila} to obtain a map $g: \R^d \to \R^n$ (this is a graph map), so that $g(\R^d)$ is $C^{1, \alpha}$. Set $G_\alpha:=\Pi(K_\alpha)$. 
Just like in the Lipschitz case, we find a Whitney decomposition of $B_d(0,1)\setminus G_\alpha$, which we denote by $\Ww_d$. We also define 
\begin{align*}
    & E_\alpha := \bigcap_{x \in K_\alpha} \overline{X_\alpha(x)}, \mbox{ and }\\
    & F_\alpha:= E_\alpha \setminus K_\alpha. 
\end{align*}
For each $S \in \Ww_d$, we define $T_S:= \Pi^{-1}(S) \cap E_0$, and then set $\Ww_n := \{T_S \, |\, S \in \Ww_d\}$. 

\begin{remark}\label{r:also-alpha}
It is immediate to see that \eqref{eq:Tswhitney0} and \eqref{eq:Tswhitney1} hold in this case, too. Lemmas \ref{l:Gamma-in-Eplus} and \ref{l:punto-bordo} and \ref{l:Eplus} are still valid, with the same proofs. 
\end{remark}
\begin{lemma}
For $T_S \in \Ww_n$, we have
\begin{align}
    & \diam(T_S) \approx_{\Lip(f)} \diam(S) \mbox{ and } \label{e:diamTS-diamS-alpha}\\
    & \dist(T_S, K_\alpha) \approx \diam(S).\label{e:distTK-alpha}
\end{align}
\end{lemma}
\begin{proof}
Both equation \eqref{e:diamTS-diamS-alpha} and \eqref{e:distTK-alpha} follow from Lemma \ref{l:diamTS}, since a point admitting an $\alpha$-paraboloid is in particular a cone point, and $f$, being $C^{1,\alpha}$, is in particular Lipschitz. 
\end{proof}

The improvement that we get with respect to the Lipschitz case, is in the following lemma.
\begin{lemma}
Let $T_S \in \Ww_n$ \textup{(}so $S \in \Ww_d$\textup{)}. Then if $y \in T_S$,
\begin{align}\label{e:dist-alpha}
    \dist(y, \Gamma) \lesssim \diam(S)^{1+\alpha}.
\end{align}
\end{lemma}
\begin{proof}
If $y \in T_S$, then $y \in F_\alpha$.
Recall that $F_\alpha = \cup_{S \in \Ww_d} T_S$. We claim that 
\begin{align}\label{e:g-F}
    g \left( \bigcup_{S \in \Ww_d} S \right) \subset F_\alpha. 
\end{align}
Indeed, suppose that for some $p \in S$, we have $g(p) \in B(0,1) \setminus F_\alpha$. Since $p \notin K_\alpha$, then $g(p) \in B(0,1) \setminus E_\alpha$. Hence, $f(p) \in X_\alpha^c(x)$ for some $x \in K_\alpha$. This contradicts the fact that $g$ is the graph map of a $C^{1,\alpha}$ extension of $f$. Thus \eqref{e:g-F} holds. 
We also claim that 
\begin{align*}
    g(S) \subset T_S
\end{align*}
for each $S \in \Ww_d$. We know that $g(S) \subset F_\alpha$. But by definition, if $g(S)\cap T_R$, for some $S\neq R \in \Ww_d$, then $g$ maps a point $p \in S\subset \R^d$ to a point $x \in \R^n$ with $\Pi(x) \neq p$. This is impossible since $g$ is the graph map of a function whose range is the orthogonal complement of $\R^d$. It is then clear that $\Pi|_\Gamma(T_S) = S$. 

Now let $y \in T_S$. Then $\Pi(y) \in S$. By the above, there exists a point $x \in \Gamma$ so that $\Pi(x)=\Pi(y)$. Thus, $|x-y| = |Q(x)-Q(y)|$, where $Q$ is the orthogonal projection onto $(\R^d)^\perp$. Since $p=\Pi(x)=\Pi(y) \in S$, then there is a point $q \in G_\alpha$ with $|q-p| \approx \diam(S)$. Moreover, we have $x,y \in X_\alpha(z)$, where $z=g(q)$. Thus,
\begin{align*}
    |Q(x) - Q(y)| &\leq |Q(x) - Q(z)| + |Q(z) - Q(y)| \\\lesssim_\theta  &|\Pi(x-z)|^{1+\alpha} + |\Pi(y-z)|^{1+\alpha}
     \approx_\theta |q-p|^{1+\alpha} \approx_\theta \diam(S)^{1+\alpha}.
\end{align*}

%Let $x \in K_\alpha$ be a point as in Lemma \ref{l:punto-bordo} (which applies in the current situation by Remark \ref{r:also-alpha}). Let $p = \Pi(x)$. We have that for any $\delta>0$, $N_\delta(p)\cap S \neq \emptyset$ for some $S \in \Ww_d$. Indeed, suppose that for some small $\delta'>0$, $N_\delta(p) \subset G_\alpha$. Let $p' \in N_{\delta'}(p) \cap G_\alpha$. Then there is a line segment $I=[p, p']$ so that any point $q \in I$ has $g(q) \in X_\alpha(x)$. But then $(x, p] \cap X_\alpha(g(q))$ for some $q \in I$, and this contradicts that $(x, p] \subset F_\alpha$. 

%By continuity of $g$, and the fact that $g$ is the graph map of a $C^{1,\alpha}$ function, 
\end{proof}
 %$\beta_E^{d,p} \approx \overline{\beta}_E^{d,p}$, in the following proofs we will not distinguish between the two.

Let $Q \in \dS$. Let $B^Q$ be a ball centered on $K_\alpha$ so that $r(B^Q) \approx \ell(Q)$ and $B_Q \subset B^Q \subset 3B_Q$. Then $\beta_E^{d,p}(Q) \approx \beta_E^{d,p}(B^Q)$.
Note that 
\begin{align*}
    E \subset E_\alpha,
\end{align*}
since if $x \in E\setminus E_\alpha$, then there exists a point $y \in K_\alpha$ so that $x \in X_\alpha^c(y)$, which contradicts $y \in K_\alpha$. Let $L_{2B^Q}$ be the plane infimising $\overline{\beta}_{\Gamma}^{d,p}(2B^Q)$. By containment, we then have
\begin{align*}
    \overline{\beta}_{E}^{d,p} (B^Q) \leq \overline{\beta}_{E}^{d,p} (B^Q, L_{2B^Q}) \leq  \overline{\beta}_{E_\alpha}^{d, p}(B^Q, L_{2B^Q}).
\end{align*}
Moreover, by Lemma \ref{l:Eplus} (and Remark \ref{r:also-alpha}), $E_\alpha$ is $d$-LCR. Further, since $x_B \in K_\alpha \subset E_\alpha$, then $x_B \in \Gamma$. Moreover, $\Gamma$ is $d$-LCR. Thus, we can apply Lemma
\ref{lemma:azzamschul} with $E_1=E_\alpha$ and $E_2= \Gamma$. We have
\begin{align*}
    \overline{\beta}_{E}^{d,p}(Q) \approx \overline{\beta}_{E}^{d,p}(B) \leq  \overline{\beta}_{E_\alpha}^{d, p} (B, L_{2B}) & \lesssim \overline{\beta}_{\Gamma}^{d,p}(2B) + \left( \frac{1}{\ell(Q)^d}\int_{E_\alpha \cap 4B} \left( \frac{\dist(y, \Gamma)}{\ell(Q)} \right)^p \, d \hdc(y) \right)^{\frac{1}{p}}\\
    & =: \beta_{\Gamma}^{d,p}(2B) + E(Q, E_\alpha, \Gamma).
\end{align*}
The sum $\sum_{Q \in \dS} \frac{\beta_{\Gamma}^{d,p}(2B^Q)^2}{\ell(Q)^{2\alpha}} \ell(Q)^d$ can be shown to be finite as in \eqref{eq:smallLip}.
On the other hand, we have
\begin{align*}
     E(Q, E_\alpha, \Gamma) 
    & =\left( \ell(Q)^{-d} \sum_{\substack{S \in \Ww_d \\ T_S \cap 4B^Q \neq \emptyset}} \int_{T_S\cap E_\alpha} \left( \frac{\dist(y, \Gamma)}{\ell(Q)}\right)^p \hdc(y) \right)^{\frac{1}{p}} \\
    & \leq \left(\ell(Q)^{-d}\sum_{\substack{S \in \Ww_d \\ T_S \cap 4B^Q \neq \emptyset}} \int_{T_S} \left( \frac{\diam(S)^{1+\alpha}}{\ell(Q)} \right)^p d \hdc(y) \right)^{\frac{1}{p}} \\
    & \lesssim \sum_{\substack{S \in \Ww_d \\ T_S \cap 4 B^Q \neq \emptyset }} \frac{\hdc(T_S) \diam(S)^{p(1+\alpha)}}{\ell(Q)^{p+d}}.
\end{align*}
Then we have
\begin{align*}
    \sum_{Q \in \dS} E(Q, E_\alpha, \Gamma) \ell(Q)^{d-2\alpha} & \lesssim \sum_{Q \in \dS} \ell(Q)^{d-2\alpha} \left( \sum_{\substack{S \in \Ww_d \\ T_S \cap 4 B^Q \neq \emptyset }} \frac{\hdc(T_S) \diam(S)^{p(1+\alpha)}}{\ell(Q)^{p+d}}\right)^{\frac{2}{p}} \\
    & \lesssim \sum_{Q \in \dS} \sum_{\substack{S \in \Ww_d \\ T_S \cap 4 B^Q \neq \emptyset }} \left( \frac{\diam(S)^{d+p(1+\alpha)}}{\ell(Q)^{p+d}} \right)^{\frac{2}{p}} \ell(Q)^{d-2\alpha}\\
    & \lesssim \sum_{S \in \Ww_d} \sum_{\substack{Q \in \dS \\ 4B^Q\cap T_S \neq \emptyset}} \frac{ \diam(S)^{\frac{2d}{p}+2(1+\alpha)}}{\ell(Q)^{2(1+\alpha) - d + \frac{2d}{p}}}
\end{align*}
It is easy to see that this sum turns out to be a geometric series: as in \eqref{e:summing1}, \eqref{e:summing2} and \eqref{e:summing3}, the choice of the exponents $p$ guarantees that $2(1+\alpha) - d+\frac{2d}{p}>0$. Then one can argue as in \eqref{e:new1}, \eqref{e:summing4} and \eqref{e:summing5}. This estimates, together with Lemmas \ref{lemma:largecubes} and \eqref{lemma:pointcubes}, conclude the proof of Lemma \ref{lemma:betaontan-alpha} and thus of Proposition  \ref{p:tangent-beta-alpha}.
One can then obtain one direction of Theorem \ref{t:main2} as in Corollary \ref{corol:final}.

\bibliographystyle{halpha-abbrv}
\bibliography{bibliography}

\Addresses


%\bibliographystyle{amsplain}

\end{document}